\documentclass[12pt]{article}

\usepackage{amscd}
\usepackage{amsmath}
\usepackage{amssymb}
\usepackage{amstext}
\usepackage{amsthm}
\usepackage{bbold}
\usepackage{bm}
\usepackage{booktabs}
\usepackage{color}
\usepackage{colonequals}
\usepackage{comment}

\usepackage{easybmat}
\usepackage{etex}
\usepackage{enumitem}
\usepackage{framed}
\usepackage{authblk}
\usepackage[dvips,letterpaper,margin=1in]{geometry}
\usepackage{graphicx}
\usepackage{listings}
\usepackage{longtable}
\usepackage{makecell}
\usepackage{mathtools}
\usepackage{multirow}
\usepackage{rotating}
\usepackage{setspace}
\usepackage[group-separator={,}]{siunitx}
\usepackage{diagbox}
\usepackage{tabu}
\usepackage{verbatim}

\definecolor{cblack}{rgb}{0,0,0}
\definecolor{cblue}{rgb}{0.121569,0.466667,0.705882}    
\definecolor{corange}{rgb}{1.000000,0.498039,0.054902}  
\definecolor{cgreen}{rgb}{0.172549,0.627451,0.172549}   
\definecolor{cred}{rgb}{0.839216,0.152941,0.156863}     
\definecolor{cpurple}{rgb}{0.580392,0.403922,0.741176}  
\definecolor{cbrown}{rgb}{0.549020,0.337255,0.294118}   
\definecolor{cpink}{rgb}{0.890196,0.466667,0.760784}
\definecolor{cgray}{rgb}{0.498039,0.498039,0.498039}
\definecolor{cgreen2}{rgb}{0.7372549019607844, 0.7411764705882353, 0.13333333333333333}

\usepackage{hyperref}
\hypersetup{
  linkcolor  = cblue,
  citecolor  = cgreen,
  urlcolor   = corange,
  colorlinks = true,
}

\newtheorem{theorem}{Theorem}[section]
\newtheorem{remark}[theorem]{Remark}

\newtheorem{lemma}[theorem]{Lemma}

\newtheorem{definition}[theorem]{Definition}

\newtheorem{proposition}[theorem]{Proposition}
\newtheorem{corollary}[theorem]{Corollary}
\newtheorem{conjecture}[theorem]{Conjecture}
\theoremstyle{plain} 
\newcommand{\thistheoremname}{}
\newtheorem*{genericthm}{\thistheoremname}

\newcommand{\Erdos}{Erd\H{o}s}

\newcommand{\Lovasz}{Lov\'{a}sz}

\newcommand{\Renyi}{R\'{e}nyi}

\newcommand{\what}{\widehat}
\newcommand{\wtil}{\widetilde}

\renewcommand{\emptyset}{\varnothing}

\makeatletter
\def\moverlay{\mathpalette\mov@rlay}
\def\mov@rlay#1#2{\leavevmode\vtop{%
   \baselineskip\z@skip \lineskiplimit-\maxdimen
   \ialign{\hfil$\m@th#1##$\hfil\cr#2\crcr}}}
\newcommand{\charfusion}[3][\mathord]{
    #1{\ifx#1\mathop\vphantom{#2}\fi
        \mathpalette\mov@rlay{#2\cr#3}
      }
    \ifx#1\mathop\expandafter\displaylimits\fi}
\makeatother

\newcommand{\CC}{\mathbb{C}}

\newcommand{\FF}{\mathbb{F}}

\newcommand{\RR}{\mathbb{R}}
\renewcommand{\SS}{\mathbb{S}}
\newcommand{\ZZ}{\mathbb{Z}}
\DeclareSymbolFont{bbold}{U}{bbold}{m}{n}
\DeclareSymbolFontAlphabet{\mathbbold}{bbold}
\newcommand{\One}{\mathbbold{1}}

\newcommand{\be}{\bm e}

\newcommand{\bx}{\bm x}

\newcommand{\bA}{\bm A}
\newcommand{\bB}{\bm B}

\newcommand{\bD}{\bm D}
\newcommand{\bE}{\bm E}

\newcommand{\bH}{\bm H}

\newcommand{\bP}{\bm P}

\newcommand{\bS}{\bm S}

\newcommand{\bV}{\bm V}

\newcommand{\one}{\bm{1}}

\newcommand{\sI}{\mathcal{I}}

\newcommand{\Mod}[1]{\ (\mathrm{mod}\ #1)}

\DeclareSymbolFont{sfoperators}{OT1}{cmss}{m}{n}
\DeclareSymbolFontAlphabet{\mathsf}{sfoperators}
\makeatletter
\renewcommand{\operator@font}{\mathgroup\symsfoperators}
\makeatother

\DeclareMathOperator{\sym}{sym}

\DeclareMathOperator{\poly}{poly}
\DeclareMathOperator{\Tr}{Tr}
\DeclareMathOperator{\rank}{rank}

\DeclareMathOperator{\Unif}{Unif}

\newcommand{\Ber}{\mathsf{Ber}}

\newcommand{\MANOVA}{\mathsf{MANOVA}}

\newcommand{\KM}{\mathsf{KM}}

\DeclareMathOperator*{\mindeg}{min\,deg}
\DeclareMathOperator*{\maxdeg}{max\,deg}

\newcommand{\eps}{\varepsilon}

\newcommand\numberthis{\addtocounter{equation}{1}\tag{\theequation}}

\title{Spectral pseudorandomness and the road to improved clique number bounds for Paley graphs}
\date{March 29, 2023}

\usepackage{authblk}
\author{Dmitriy Kunisky\thanks{Email: \texttt{dmitriy.kunisky@yale.edu}. Partially supported by ONR Award N00014-20-1-2335 and a Simons Investigator Award to Daniel Spielman.}}
\affil{Department of Computer Science, Yale University}

\begin{document}

\maketitle

\thispagestyle{empty}

\begin{abstract}
    We study subgraphs of Paley graphs of prime order $p$ induced on the sets of vertices extending a given independent set of size $a$ to a larger independent set.
    Using a sufficient condition proved in the author's recent companion work, we show that a family of character sum estimates would imply that, as $p \to \infty$, the empirical spectral distributions of the adjacency matrices of any sequence of such subgraphs have the same weak limit (after rescaling) as those of subgraphs induced on a random set including each vertex independently with probability $2^{-a}$, namely, a Kesten-McKay law with parameter $2^a$.
    We prove the necessary estimates for $a = 1$, obtaining in the process an alternate proof of a character sum equidistribution result of Xi~(2022), and provide numerical evidence for this weak convergence for $a \geq 2$.
    We also conjecture that the minimum eigenvalue of any such sequence converges (after rescaling) to the left edge of the corresponding Kesten-McKay law, and provide numerical evidence for this convergence.
    Finally, we show that, once $a \geq 3$, this (conjectural) convergence of the minimum eigenvalue would imply bounds on the clique number of the Paley graph improving on the current state of the art due to Hanson and Petridis~(2021), and that this convergence for all $a \geq 1$ would imply that the clique number is $o(\sqrt{p})$.
\end{abstract}

\clearpage

\thispagestyle{empty}

\tableofcontents

\clearpage

\setcounter{page}{1}
\pagestyle{plain}

\section{Introduction}

Let $p \equiv 1 \Mod{4}$ be prime.
Let $\FF_p$ be the finite field of order $p$, identified with $\ZZ / (p\ZZ)$ (the integers modulo $p$), let $\FF_p^\times$ be its multiplicative group (of non-zero elements), and let $\SS_p < \FF_p^{\times}$ be the subgroup of non-zero quadratic residues (those $x \neq 0$ such that $x \equiv y^2 \Mod{p}$ for some $y$).
The \emph{Paley graph}, defined on vertex set $\FF_p$, has $x \sim y$ if and only if $x - y \in \SS_p$.

The graphs $G_p$ are thought to be \emph{pseudorandom}, behaving in many ways like \Erdos-\Renyi~(ER) random graphs with edge probability $\frac{1}{2}$, this probability chosen to match $G_p$ being $(\frac{p - 1}{2})$-regular.
For example, the $G_p$ satisfy the ``quasi-randomness'' conditions described in \cite{CGW-1989-QuasiRandomGraphs}; in particular, the extreme eigenvalues scale roughly similarly to those (typically) of an ER graph, and the number of appearances of any fixed graph $G^{(0)}$ as an induced subgraph in $G_p$ as $p \to \infty$ scales as the number of appearances of $G^{(0)}$ (typically) in an ER graph.
Going on this intuition, the $G_p$ are believed to be useful for several \emph{derandomization} problems.
Namely, it is believed that deterministic matrices satisfying the restricted isometry property can be constructed from the closely related \emph{Paley equiangular tight frames} \cite{BFMW-2013-DeterministicRIP,BMM-2017-ConditionalRIP}, and the $G_p$ themselves are believed to be explicit constructions of tight lower bounds on the diagonal Ramsey numbers.
However, it seems difficult to prove such claims, as they require much more precise number-theoretic analysis than the above notions of pseudorandomness.

In this paper, we study new proof techniques for controlling the quantities involved in the latter Ramsey theory application.\footnote{See Appendix~\ref{app:rip} for discussion of limitations of our approach for studying the restricted isometry property.}
For a graph $G$, we write $\omega(G)$ for its clique number, $\overline{G}$ for its complement, and $\alpha(G) \colonequals \omega(\overline{G})$ for its independence number.
For the purposes of Ramsey theory, we are interested in showing that $\omega(G_p)$ and $\alpha(G_p)$ are both small (see Section~\ref{sec:related} for more discussion).
For the Paley graph, $G_p$ is isomorphic to $\overline{G_p}$, so $\alpha(G_p) = \omega(G_p)$ and there is only one number to control.
The literature usually discusses the clique number of the Paley graph, but here we will discuss the independence number instead, both for mild notational convenience later and to draw an analogy with similar problems below.

Consequences for Ramsey theory aside, determining $\alpha(G_p)$ is a prominent problem in number theory; for example, $\alpha(G_p)$ is an upper bound on the smallest quadratic non-residue modulo $p$, determining which is another long-standing open problem.
Through this connection, \cite{GR-1990-LowerBoundsLeastQNR} showed that, for infinitely many primes $p$,
\begin{equation}
    \alpha(G_p) \geq \log p \log\log\log p.
\end{equation}
This is slightly larger than the typical value for an ER graph on $p$ vertices, which is of order $\Theta(\log p)$.
Conditional on the Generalized Riemann Hypothesis, the $\log\log\log p$ term can also be improved to $\log \log p$ \cite[Theorem 13.5]{Montgomery-1971-MultiplicativeNumberTheory}, and this is compatible with a more detailed random model of $G_p$ \cite{Mrazovic-2017-RandomModelPaleyGraph}.
More generally, based on such results as well as numerical computations of $\alpha(G_p)$, the following conjecture is widely believed.
\begin{conjecture}[\cite{Shearer-1986-LowerBoundsDiagonalRamsey,BMR-2013-SquaresDifferenceSets,Yip-2022-CliqueNumberPaleyPrimePower,KY-2022-PaleyDegree4SOS}]
    \label{conj:clique-polylog}
    $\alpha(G_p) = O(\mathsf{polylog}(p))$.
\end{conjecture}
\noindent
As we detail in Section~\ref{sec:related}, proving this would yield a substantial improvement on the best known explicit construction of Ramsey graphs (ones giving strong lower bounds on the diagonal Ramsey numbers).

In stark contrast to Conjecture~\ref{conj:clique-polylog}, the best known upper bound on $\alpha(G_p)$, due to Hanson and Petridis, is as follows.
\begin{theorem}[\cite{HP-2021-SumsetsPaleyClique,BSW-2021-ProductDirectionsGaloisPaley}]
    \label{thm:hp-bound}
    $\alpha(G_p) \leq \frac{\sqrt{2p - 1} + 1}{2} \leq \frac{1}{\sqrt{2}}\sqrt{p} + 1$.
\end{theorem}
\noindent
This is only a modest improvement on the ``trivial'' bound $\alpha(G_p) \leq \sqrt{p}$, which has many proofs: $\sqrt{p}$ is the value of both the spectral ``Hoffman bound'' and the \Lovasz\ $\vartheta$ function bounds on $\alpha(G_p)$ \cite[Theorem 8]{Lovasz-1979-ShannonCapacityGraph}, and one may also obtain the same bound combinatorially \cite[Lemma 1.2]{Yip-2022-CliqueNumberPaleyPrimePower}.

\subsection{Relaxation and localization: two bounding strategies}

Beyond our lack of tight quantitative bounds on $\alpha(G_p)$, this state of affairs is conceptually disappointing because the known proofs of Theorem~\ref{thm:hp-bound} use in a deep way the algebraic structure of $G_p$.
It would be more satisfying to treat $\alpha(G_p)$ in the same framework as other ``packing'' or independence number problems arising in extremal combinatorics and coding theory.
In particular, \emph{convex relaxation} bounds on $\alpha(G_p)$---bounds by linear (LP) or semidefinite (SDP) programs---have been useful for bounding independence numbers of many other graphs of interest \cite{Delsarte-1972-LPUnrestrictedCodes,Lovasz-1979-ShannonCapacityGraph,CE-2003-SpherePackingLP, Schrijver-2005-CodeBoundsSDP, BV-2008-KissingNumberSDP, LV-2015-SDPHierarchyPacking}.
Can techniques based on convex relaxation improve on the state-of-the-art bound on $\alpha(G_p)$?

It is unclear whether to expect this approach to succeed.
On the one hand, a long line of work has shown that, for ER random graphs $G$, increasingly sophisticated convex relaxations of $\alpha(G)$ do \emph{not} improve substantially on simple spectral bounds (which may be viewed as very simple instances of convex relaxation).
Perhaps most prominent among these results are those showing such lower bounds for the otherwise quite powerful \emph{sum-of-squares (SOS)} hierarchy of semidefinite programs \cite{MW-2013-PlantedClique,MPW-2015-PlantedClique, RS-2015-SOS4PlantedClique, DM-2015-SOSPlantedClique, HKPRS-2018-PlantedCliqueSOS4, BHKKMP-2019-PlantedClique, Pang-2021-PlantedCliqueSOSExact}.
Thus, if we believe in sufficiently strong pseudorandomness of the Paley graph, we should not expect convex relaxations of $\alpha(G_p)$ to improve much on the $\sqrt{p}$ bound.
On the other hand, the work of \cite{GLV-2009-BlockDiagonalSDPPaleyClique} and the author's recent results with Yu \cite{KY-2022-PaleyDegree4SOS} suggest that such convex relaxations (in particular, the degree~4 SOS relaxation) might after all improve on the bound of Theorem~\ref{thm:hp-bound} and even prove that $\alpha(G_p) = O(p^{\frac{1}{2} - \eps})$ for some $\eps > 0$.
A possible reason for this is that the spectrum of $G_p$ is sufficiently different from that of an ER random graph: the eigenvalues of an ER graph have the semicircle distribution of random matrix theory, while $G_p$ is strongly regular and has only three distinct eigenvalues, a very differently-shaped empirical spectral distribution.\footnote{See the discussion of graph matrix norm bounds in \cite{KY-2022-PaleyDegree4SOS} for how to convert this discrepancy into evidence that higher degrees of SOS may give improved bounds on $\alpha(G_p)$.}

The recent work of \cite{MMP-2019-LPCliquesPaley} contributed a complementary idea for bounding $\alpha(G_p)$, which we call the \emph{localization} approach.\footnote{The localization approach is somewhat similar to the tools developed by Elias-Bassalygo and Kalai-Linial for bounding the independence numbers of interest in coding theory \cite{Bassalygo-1965-UpperBoundsCodes, Berlekamp-1968-AlgebraicCodingTheory, KL-1995-DistanceDistributionCodes}.}
Let $\sI_a(G)$ denote the set of independent sets of size $a$ in a graph $G$.
The idea is to bound $\alpha(G)$ by enumerating all $I \in \sI_a(G)$, and then to consider the independence number of graphs induced on all vertices extending $I$ to a larger independent set.
The basic tool is the following.
\begin{definition}[Localization]
    For $X \subseteq V(G_p)$, let $G_{p, X}$ denote the subgraph of $G_p$ induced on the vertices not adjacent to any vertex of $X$ (not including vertices in $X$ itself).
    We call $G_{p, X}$ a \emph{localization} of $G$ of \emph{degree} $|X|$.
\end{definition}

\begin{proposition}[Localization bound]
    \label{prop:loc}
    For any $a \leq \alpha(G_p)$,
    \begin{equation}
        \label{eq:loc}
        \alpha(G_p) = a + \max_{I \in \sI_a(G_p)} \omega(G_{p, I}).
    \end{equation}
\end{proposition}
If one has some general upper bound $\alpha(G) \leq f(G)$ for all graphs $G$, then instead of bounding $\alpha(G_p) \leq f(G_p)$, one may bound by substituting $f$ into the right-hand side of \eqref{eq:loc}:
\begin{equation}
    \label{eq:loc-bound}
    \alpha(G_p) \leq a + \max_{I \in \sI_a(G_p)} f(G_{p, I}).
\end{equation}
For example, \cite{MMP-2019-LPCliquesPaley} took $f(G)$ to be the \Lovasz\ function $\vartheta(G)$, an SDP relaxation of $\alpha(G)$ equivalent to the degree 2 SOS relaxation, with $a = 1$.
It turns out that the graphs $G_{p, \{x\}}$ are isomorphic for all $x \in \FF_p$ (see Corollary~\ref{cor:transitivity}), so the method of \cite{MMP-2019-LPCliquesPaley} only requires looking at one graph:
\begin{equation}
    \label{eq:mmp-lp-bound}
    \alpha(G_p) \leq 1 + \vartheta(G_{p, \{0\}}).
\end{equation}
The graph $G_{p, \{0\}}$ is still highly symmetric, in particular having a circulant adjacency matrix, so $\vartheta(G_{p, \{0\}})$ may be reduced to an LP and solved efficiently for large $p$.
Doing this, \cite{MMP-2019-LPCliquesPaley} found empirically that the bound \eqref{eq:mmp-lp-bound} usually almost exactly matches that of Theorem~\ref{thm:hp-bound}.
Localization therefore improves the quality of the bound we obtain, from the trivial bound to one matching the state of the art:
\begin{align}
  \alpha(G_p) &\leq \vartheta(G_p) = \sqrt{p} \,\,\, \text{without localization,} \\
  \alpha(G_p) &\leq 1 + \vartheta(G_{p, \{0\}}) \approx \frac{1}{\sqrt{2}} \sqrt{p} \,\,\, \text{with degree 1 localization (empirically).}
\end{align}

Combining convex relaxation and localization, we have a family of strategies for bounding $\alpha(G_p)$ with two degrees of freedom: we may choose a degree of localization $a$, choose a convex relaxation bound $f$, and try to analyze the right-hand side of \eqref{eq:loc-bound}.
For fixed $a$, strengthening the bound $f$ will improve the result.
Extrapolating from the results of \cite{MMP-2019-LPCliquesPaley}, it seems reasonable to expect that, for fixed $f$, increasing the degree of localization $a$ will also improve the result.
But how specifically do these two choices interact?
To understand this, we must understand the localizations $G_{p, I}$ in greater detail.

\subsection{Random model for Paley graph localizations}

The main thesis of this paper is that the localizations $G_{p, I}$ have a collection of properties we call \emph{spectral pseudorandomness}.
The idea is that the spectrum of $G_{p, I}$ behaves like that of an induced subgraph on a random subset of vertices of $G_p$ of comparable size.
Note that the spectrum of $G_p$ itself, as noted above, is \emph{not} close to that of an ER graph of the same size.
In introducing the notion of spectral pseudorandomness, we are proposing a more fine-grained way in which pseudorandomness \emph{does} appear in the spectral structure of $G_p$.

\begin{definition}[Random Paley subgraphs]
    Let $H_{p, \beta}$ be the random graph formed as follows: form a subset $W \subseteq \FF_p$ randomly by including every vertex independently with probability $\beta$, and let $H_{p, \beta}$ be the subgraph induced by $G_p$ on $W$.
\end{definition}

\noindent
We then propose that $G_{p, I}$ behaves like $H_{p, \beta}$ with $\beta = 2^{-|I|}$.
The intuition behind the choice of $\beta$ is that $x \in V(G_{p, I})$ if $x$ is not adjacent to each $y \in I$.
If $G_p$ is pseudorandom, these non-adjacencies may be viewed as $|I|$ ``independent'' events that each occur with ``probability'' $\frac{1}{2}$, for a total ``probability'' of $2^{-|I|}$ that any given $x \in V(G_{p, I})$.

The \emph{empirical spectral distribution (e.s.d.)}\ of a matrix $\bA \in \RR^{n \times n}_{\sym}$ is the probability measure $\frac{1}{n}\sum_{i = 1}^n \delta_{\lambda_i(\bA)}$.
The e.s.d.\ of a graph is the e.s.d.\ of its adjacency matrix.
In this language, we claim that, for large $p$ and small $|I|$, the (deterministic) e.s.d.\ of $G_{p, I}$ is close to the typical (random) e.s.d.\ of $H_{p, 2^{-|I|}}$.
Below we will give a collection of conjectures describing this closeness in different quantitative senses.
Strong enough such statements, controlling the minimum eigenvalue of $G_{p, I}$, will let us apply the localization strategy effectively with $f(G)$ a spectral bound on $\alpha(G)$.

The following recent result, first proposed as part of a collection of conjectures by \cite{HZG-2017-MANOVA}, describes the limiting e.s.d.\ of the $H_{p, 2^{-|I|}}$.
We will argue that the e.s.d.'s of the $G_{p, I}$ are close to these measures.
\begin{definition}[Kesten-McKay measure]
    For $v \geq 1$, the Kesten-McKay measure with parameter $v$, denoted $\KM(v)$, is the probability measure on $\RR$ with density
    \begin{equation}
        \One\{|x| \leq 2\sqrt{v - 1}\} \frac{v\sqrt{4(v - 1) - x^2}}{2\pi(v^2 - x^2)}\, dx + \frac{\max\{2 - v, 0\}}{2}\delta_{-v}(x) + \frac{\max\{2 - v, 0\}}{2}\delta_{v}(x),
    \end{equation}
    where $\delta_c$ denotes a Dirac mass at $c$.
\end{definition}
\noindent
Let us write $\bA_G$ for the adjacency matrix of a graph $G$.
\begin{theorem}[\cite{MMP-2019-RandomSubensembles,Kunisky-2023-MANOVAProductProjections}]
    The (random) e.s.d.\ of $\frac{2}{\beta\sqrt{p}}\bA_{H_{p, \beta}}$ converges weakly in probability\footnote{Random measures $\mu_n$ converge weakly in probability to a deterministic measure $\mu$ if $\mu_n(A) \to \mu(A)$ in probability for every measurable set $A$.} to $\KM(1 / \beta)$ as $p \to \infty$.
\end{theorem}
\noindent
A direct combinatorial proof of a minor variation on this was first given by \cite{MMP-2019-RandomSubensembles}, while \cite{Kunisky-2023-MANOVAProductProjections} later observed that this result is a consequence of a more general asymptotic freeness (in the sense of free probability) phenomenon.
We will eventually exploit this connection in our arguments; see the discussion around Conjecture~\ref{conj:asymp-free}.

\subsection{A relaxation-localization tradeoff?}

Before proceeding to our specific conjectures and results, let us comment generally on what spectral pseudorandomness might imply for our earlier question of how relaxation and localization interact.

As $v \to \infty$, the measures $\KM(v)$ converge, after rescaling, to the semicircle law (see Figure~\ref{fig:example-loc} for an illustration).
Thus, a consequence of spectral pseudorandomness is that, as the degree of localization increases, the localizations of $G_p$ more and more resemble (at the level of the spectrum) ER random graphs.

But, as mentioned earlier, on ER random graphs, stronger convex relaxations of $\alpha(G)$ are essentially no more effective than weaker relaxations or a basic spectral bound.
It is perhaps reasonable to extrapolate this to the claim that, the more $G$ resembles an ER random graph at the level of the spectrum, the less ``return on investment'' we get from bounding $\alpha(G) \leq f(G)$ with a more computationally expensive convex relaxation $f(G)$.

We thus arrive at the following intuitive idea: increasing the convex relaxation strength is more helpful at low degrees of localization and less helpful at higher degrees of localization; increasing the degree of localization is more helpful for weak convex relaxations and less helpful for strong convex relaxations.
If this is true, we should be able to ``trade'' higher degrees of localization for weaker convex relaxations while obtaining comparable bounds.
This would be convenient because, for proving actual results, it seems hard to analyze sophisticated convex relaxations applied to deterministic pseudorandom graphs (as attempted for SOS relaxations on Paley graphs by \cite{MMP-2019-LPCliquesPaley, KY-2022-PaleyDegree4SOS}).

In Table~\ref{tab:relaxation-localization}, we present the combinations of relaxation and localization that have been examined before for Paley graphs and those that we look at in this paper.
To summarize, prior work has examined a few SOS relaxations with low degrees of localization, while we mostly restrict our attention to the weaker spectral bound but consider arbitrary degrees of localization.
We focus on this case because the spectrum of $G_{p, I}$, characterized by spectral pseudorandomness, already gives enough information to compute the spectral bound.
And we will show that, conditional on natural conjectures, even spectral bounds can prove strong bounds on $\alpha(G_p)$ provided we combine them with high degrees of localization.
We leave the task of filling in the missing entries in Table~\ref{tab:relaxation-localization}---characterizing how more sophisticated convex relaxations interact with localization---to future work.

\begin{table}
    \begin{center}
        \renewcommand{\arraystretch}{1.4}
        \begin{tabular}{|c|l|cccc|}
          \hline
      \multicolumn{2}{|c|}{} & \multicolumn{4}{c|}{Localization degree $a$} \\
      \cline{3-6}
      \multicolumn{2}{|c|}{} & 0 & 1 & 2 & $\geq 3$ \\
      \hline
      \multirow{4}{*}{\rotatebox{90}{Bound $f(G)$\,\,}} & Spectral & folklore\,$^{(\text{p})}$ & this paper\,$^{(\text{p})}$ & this paper\,$^{(\text{c})}$ & this paper\,$^{(\text{c})}$ \\
      & SOS degree 2 & implicitly \cite{Lovasz-1979-ShannonCapacityGraph}\,$^{(\text{p})}$ & \cite{MMP-2019-LPCliquesPaley}\,$^{(\text{e})}$ & this paper\,$^{(\text{e})}$ & ? \\
      & SOS degree 4 & \cite{KY-2022-PaleyDegree4SOS}\,$^{(\text{p}, \text{e})}$ & ? & ? & ? \\
                             & SOS degree $\geq 6$ & ? & ? & ? & ? \\
                                                                 \hline
        \end{tabular}
\end{center}
\caption{\textbf{Known relaxation-localization bounds.} We summarize the analyses of combinations of localization and convex relaxation that have been performed for bounding $\alpha(G_p) = \omega(G_p)$ in the literature. We mark those works that \textbf{p}rove results on these bounds by $(\text{p})$, those works that prove results conditional on \textbf{c}onjectures by $(\text{c})$, and those works that give \textbf{e}mpirical or numerical evidence by $(\text{e})$. We note that degree 2 SOS is equivalent to the \Lovasz\ $\vartheta$ function, and that we examine degree 2 localization combined with this relaxation in Appendix~\ref{app:theta-01}.
}
        \label{tab:relaxation-localization}
\end{table}

\subsection{Spectral pseudorandomness conjectures}

We now present our specific conjectures about spectral pseudorandomness.
We will phrase all of our conjectures in terms of the degree $a \geq 1$ of the localizations involved.

Our first conjecture states our claim of spectral pseudorandomness at the level of weak convergence.
\begin{conjecture}[Weak convergence, degree $a$]
    \label{conj:weak-conv}
    For all $p$ sufficiently large, let $I \in \sI_a(G_p)$.
    Then, the (deterministic) e.s.d.\ of $\frac{2^{a + 1}}{\sqrt{p}} \bA_{G_{p, I}}$ converges weakly to $\KM(2^a)$.
\end{conjecture}
\noindent
This will follow from the more technical but more powerful Conjecture~\ref{conj:asymp-free}, which states that a pair of projection matrices associated to $G_p$ and to $I$, respectively, are \emph{asymptotically free}.
This allows us to treat Conjecture~\ref{conj:weak-conv} with the tools developed by \cite{Kunisky-2023-MANOVAProductProjections}.
See Section~\ref{sec:related} for some discussion of the interest of Conjecture~\ref{conj:asymp-free} for free probability theory.

Our second conjecture gives a form of spectral pseudorandomness ``at the edge,'' concerning the minimum eigenvalue of the adjacency matrix.\footnote{It would also be reasonable to make the same conjecture for the maximum eigenvalue, though one should exclude the large isolated Perron-Frobenius eigenvalue.}

\begin{conjecture}[Minimum eigenvalue, degree $a$]
    \label{conj:min-eval}
    We have
    \begin{equation}
        \lim_{p \to \infty} \min_{I \in \sI_a(G_p)} \frac{2^{a + 1}}{\sqrt{p}} \lambda_{\min}(\bA_{G_{p, I}}) = \lim_{p \to \infty} \max_{I \in \sI_a(G_p)} \frac{2^{a + 1}}{\sqrt{p}} \lambda_{\min}(\bA_{G_{p, I}}) = -2\sqrt{2^a - 1}.
    \end{equation}
\end{conjecture}
\noindent
Note that the left-hand side has the same normalization as in Conjecture~\ref{conj:weak-conv}, while the right-hand side is the left edge of the support of $\KM(2^a)$.

\begin{remark}[Symmetry in low-degree localizations]
    The cases $a \in \{1, 2\}$ of these statements are especially simple, because for either of these values the $G_{p, I}$ are all isomorphic for any $I \in \sI_a(G_p)$ (see Corollary~\ref{cor:transitivity}).
    This symmetry for the case $a = 1$ was used already by \cite{MMP-2019-LPCliquesPaley}, but they did not take advantage of the same for the case $a = 2$.
    In Appendix~\ref{app:theta-01}, we present empirical evidence that degree~2 localization combined with the degree~2 SOS relaxation already seems to improve on the state-of-the-art bound of Theorem~\ref{thm:hp-bound}.
\end{remark}

\begin{remark}[General low-degree subgraphs]
    \label{rem:low-deg}
    Conjectures~\ref{conj:weak-conv} and \ref{conj:min-eval} are also sensible for a more general class of induced subgraph on what might be called \emph{low-degree} subsets of vertices: sets of those vertices that are either adjacent or not adjacent to each of a small set $I \subseteq V$.
    That is, $I$ need not form an independent set, and the constraints need not all demand non-adjacency to each $v \in I$.
    However, as we demonstrate empirically in Appendix~\ref{app:rip}, one \emph{cannot} extend our conjectures to arbitrary induced subgraphs of $G_p$.
\end{remark}

\subsection{Necklace character sum conjectures}

We next present one more conjecture, which describes \emph{character sum estimates} similar to ones long studied in analytic number theory which would imply our first spectral pseudorandomness conjecture.
This connection is thanks to the relationship between the Paley graph and the following function.
\begin{definition}[Legendre symbol]
    The \emph{Legendre symbol} is the function $\chi: \FF_p \to \{-1, 0, 1\} \subset \CC$ given by
    \begin{equation}
        \chi(x) \colonequals \left\{\begin{array}{rl} 0 & \text{if } x = 0, \\ 1 & \text{if } x \in \SS_p, \\ -1 & \text{otherwise}. \end{array}\right.
    \end{equation}
\end{definition}
\noindent
The crucial property of the Legendre symbol is that it is a \emph{multiplicative character} of $\FF_p^{\times}$, meaning that
\begin{align}
  \chi(1) &= 1, \\
  \chi(xy) &= \chi(x)\chi(y) \text{ for all } x, y \in \FF_p.
\end{align}

Let us write $\bS_G$ for the $\{\pm 1\}$-valued adjacency matrix of a graph $G$, i.e., having
\begin{equation}
    (\bS_G)_{xy} = \left\{\begin{array}{rl} 0 & \text{if } x = y, \\ 1 & \text{if } x \text{ is adjacent to } y \text{ in } G, \\ -1 & \text{otherwise.} \end{array} \right.
\end{equation}
Then, this matrix for the Paley graph is a circulant matrix populated with the Legendre symbol:
\begin{equation}
    (\bS_{G_p})_{xy} = \chi(x - y).
\end{equation}

\begin{definition}[Necklace character sum]
    Let $k \geq 1$ and $Z_1, \dots, Z_k \subseteq \FF_p$.
    The associated \emph{necklace character sum} is
    \begin{equation}
        \Sigma(Z_1, \dots, Z_k) \colonequals \sum_{x_1, \dots, x_k \in \FF_p} \chi(x_2 - x_1) \cdots \chi(x_k - x_{k - 1}) \chi(x_1 - x_k) \prod_{i = 1}^k \prod_{z \in Z_i} \chi (x_i - z).
    \end{equation}
\end{definition}
\noindent
Necklace character sums will appear naturally in our calculations with the e.s.d.'s of the adjacency matrices of $G_{p, I}$ because, for $\bD_1, \dots, \bD_k \in \RR^{\FF_p \times \FF_p}$ diagonal matrices with diagonal entries $(\bD_i)_{xx} \colonequals \prod_{z \in Z_i} \chi(x - z)$, we have
\begin{equation}
    \label{eq:Sigma-Tr}
    \Sigma(Z_1, \dots, Z_k) = \Tr(\bD_1 \bS_{G_p} \cdots \bD_k \bS_{G_p}).
\end{equation}

Our main conjecture about these sums, again graded by degree, is as follows.

\begin{conjecture}[Necklace character sum estimates, degree $a$]
    \label{conj:char-sum}
    For all $k \geq 1$,
    \begin{equation}
        \lim_{p \to \infty} p^{-(\frac{k}{2} + 1)}\max_{\substack{Z_1, \dots, Z_k \subseteq \FF_p \\ Z_i \neq \emptyset \\ |Z_1 \cup \cdots \cup Z_k| \leq a}} |\Sigma(Z_1, \dots, Z_k)| = 0.
    \end{equation}
\end{conjecture}
\noindent
One way to motivate the conjecture, starting from the linear-algebraic interpretation of necklace character sums \eqref{eq:Sigma-Tr}, is to note that $\|\bD_1 \bS_{G_p} \cdots \bD_k \bS_{G_p}\| \leq \|\bS_{G_p}\|^k = p^{k/2}$ (see Proposition~\ref{prop:paley-spectral}), and thus we immediately have the bound
\begin{equation}
    |\Sigma(Z_1, \dots, Z_k)| \leq p^{\frac{k}{2} + 1}.
\end{equation}
Conjecture~\ref{conj:char-sum} posits cancellations in necklace character sums beyond this ``trivial'' estimate.

\begin{remark}[Non-emptiness assumption]
    We cannot drop the assumption that the $Z_i$ are non-empty, since, by the discussion above and Proposition~\ref{prop:paley-spectral}, we have, for $k$ even,
    \begin{equation}
        \Sigma(\emptyset, \dots, \emptyset) = \Tr(\bS_{G_p}^k) = \Omega(p^{\frac{k}{2} + 1}).
    \end{equation}
\end{remark}

\begin{remark}[Tighter estimates]
    By multiplicativity of the Legendre symbol,
    \begin{equation}
        \Sigma(Z_1, \dots, Z_k) = \sum_{\bx \in \FF_p^k} \chi(q(\bx))
    \end{equation}
    for a polynomial $q$.
    This is the standard form of a multivariate multiplicative character sum; such quantities have been studied at great length in analytic number theory.
    For ``generic'' $q$, one expects ``square root cancellations'' in such sums: they should have magnitude $O(p^{k/2})$, the same order as the sum of $p^k$ i.i.d.\ random variables drawn as $\Unif(\{ \pm 1 \})$.
    However, our $q$ are not always generic enough for this to happen.
    For example, we will see in our proof of Theorem~\ref{thm:level1} in Section~\ref{sec:pf:thm:level1} that, when $Z_1 = \cdots = Z_k = \{z\}$ for some $z \in \FF_p$, then there is a ``spurious degree of freedom'' in this sum, and it can in fact be rewritten as
    \begin{equation}
        \Sigma(\{z\}, \dots, \{z\}) = \sum_{\bx \in \FF_p^k} \chi(q(\bx)) = (p - 1) \sum_{\bx \in \FF_p^{k - 1}} \chi(\wtil{q}(\bx))
    \end{equation}
    for a ``more generic'' polynomial $\wtil{q}$ in fewer variables.
    Accordingly, we will only be able to show that this sum (for fixed $a$ and $k$) is of order $O(p \cdot p^{(k - 1) / 2}) = O(p^{(k + 1) / 2})$.
    We will see in Appendix~\ref{app:level2} that the same happens when $Z_1 = \cdots = Z_k = \{z_1, z_2\}$ for some distinct $z_1, z_2 \in \FF_p$.
    It is reasonable to conjecture that there are a few special cases for which $|\Sigma(Z_1, \dots, Z_k)| = O(p^{(k + 1) / 2})$, and that, outside of those cases, $|\Sigma(Z_1, \dots, Z_k)| = O(p^{k/2})$.
\end{remark}

\subsection{Main results}

We now give some relationships among our conjectures and some preliminary results that we are able to prove.
First, we show that necklace character sum estimates imply our weak convergence conjecture.

\begin{theorem}
    \label{thm:conj-equiv}
    If Conjecture~\ref{conj:char-sum} holds at degree $a$, then Conjecture~\ref{conj:weak-conv} holds at degree $a$.
\end{theorem}

Next, we show that all of our conjectures hold at degree 1, i.e., for the graph $G_{p, \{0\}}$ induced on the subset of vertices $\FF_p^{\times} \setminus \SS_p$.
\begin{theorem}
    \label{thm:level1}
    Conjectures~\ref{conj:char-sum}, \ref{conj:weak-conv}, and \ref{conj:min-eval} all hold at degree $a = 1$.
\end{theorem}
\noindent
We note that the result of Conjecture~\ref{conj:weak-conv} at degree 1 appears, albeit in very different form, in the recent work \cite{Xi-2021-DistributionMultiplicativeKloostermanSum}.
We discuss this connection in Section~\ref{sec:equidistribution-level1}; our proof is arguably more conceptual and gives a new free probability interpretation of the Kesten-McKay law appearing in the statement.
In Appendix~\ref{app:higher-degree}, we also give some minor results towards extending the ideas from our proof of Conjecture~\ref{conj:char-sum} at degree 1 to higher degrees.

Finally, we show that our strongest conjecture on the minimum eigenvalue implies a sequence of clique number bounds whose strength increases with the degree.

\begin{theorem}
    \label{thm:clique-bound}
    If Conjecture~\ref{conj:min-eval} holds at degree $a$, then
    \begin{equation}
        \omega(G_p) = \alpha(G_p) \leq \frac{\sqrt{2^a - 1}}{2^{a - 1}} \sqrt{p} + o(\sqrt{p}).
    \end{equation}
\end{theorem}
\noindent
If an explicit error term is obtained in Conjecture~\ref{conj:min-eval}, then the $o(\sqrt{p})$ term may be made explicit as well.
Unfortunately, as we discuss in Section~\ref{sec:challenges-eval}, we believe proving Conjecture~\ref{conj:min-eval} at degree $a \geq 2$ will be significantly harder than proving Conjecture~\ref{conj:weak-conv}.
But, in Section~\ref{sec:numerical}, we give numerical evidence for both Conjectures~\ref{conj:weak-conv} and \ref{conj:min-eval} at low degrees.

Let us compare Theorem~\ref{thm:clique-bound} with the state-of-the-art bound of Hanson and Petridis (our Theorem~\ref{thm:hp-bound}).
The latter scales as $\frac{\sqrt{2}}{2}\sqrt{p} \approx 0.707 \sqrt{p}$.
For $a = 1$ the bound of Theorem~\ref{thm:clique-bound} scales as $\sqrt{p}$, the ``trivial'' bound (see Remark~\ref{rem:Gp0-edges} for an explanation).
For $a = 2$ our bound would scale as $\frac{\sqrt{3}}{2} \sqrt{p} \approx 0.866 \sqrt{p}$, in between the trivial bound and the Hanson-Petridis bound.
But, already for $a = 3$ our bound would scale as $\frac{\sqrt{7}}{4}\sqrt{p} \approx 0.661 \sqrt{p}$, improving on the Hanson-Petridis bound.
Moreover, if Conjecture~\ref{conj:min-eval} held for arbitrarily large $a$, then our bound would show that $\omega(G_p) = o(\sqrt{p})$, and taking $a = a(p)$ growing slowly should allow explicit improvements on the $\sqrt{p}$ scaling using the same proof strategy.

\subsection{Related work}
\label{sec:related}

\paragraph{Constructive Ramsey theory}
The \emph{diagonal Ramsey number} $R(k)$ is the maximum of $|V(G)|$ over graphs $G$ with $\omega(G) \leq k$ and $\alpha(G) \leq k$.
Ramsey's famous theorem initiating Ramsey theory \cite{Ramsey-1929-RamseyNumbers} gives a finite upper bound on $R(k)$, and it has since been of great interest in combinatorics to prove tighter bounds on $R(k)$.
By showing that $G$ an ER graph on $p$ vertices has $\max\{\omega(G), \alpha(G)\} \leq 2 \log_2p$ with positive probability, \Erdos\ \cite{Erdos-1947-GraphTheoryRamsey} showed, in an early use of the probabilistic method (see, e.g., \cite{Spencer-1994-TenLecturesProbabilisticMethod}), that $R(k) \geq 2^{k/2}$.
It is a long-standing open problem to match this result with an explicit construction, and the Paley graph is perhaps the main candidate for such a graph.\footnote{It is also a long-standing open problem to prove matching upper and lower bounds on $R(k)$; the main result of \cite{ES-1935-CombinatorialProblemGeometry} is that $R(k) \leq 4^k$, and the recent breakthrough \cite{CGMS-2023-ExponentialDiagonalRamsey} for the first time improves the base of the exponent.}
The best known results \cite{Cohen-2016-PolylogarithmicEntropyRamseyGraphs,CZ-2016-TwoSourceExtractorsRamsey} give deterministic constructions of $G$ on $p$ vertices with $\max\{\omega(G), \alpha(G)\} \leq \exp((\log\log p)^C)$ for some $C > 0$.\footnote{Equivalently, this is a deterministic construction for $R(k) \geq \exp(\exp((\log k)^c))$ for some $c > 0$, while Conjecture~\ref{conj:clique-polylog} would give that the Paley graph is a deterministic construction for $R(k) \geq \exp(k^c)$ for some $c > 0$. Note also that the consequence $\omega(G_p) = \alpha(G_p) = o(\sqrt{p})$ of our conjectures would not improve on this result; much stronger sub-polynomial clique number bounds would be required for that.}
See, e.g., \cite[Sections 8.2 and 9]{Wigderson-2019-MathematicsComputation} or \cite{Vadhan-2012-Pseudorandomness} for more discussion and consequences of such constructions for the theory of algorithms.

\paragraph{Generic clique number bounds}
We have restricted our attention to combining localization with general-purpose convex relaxation bounds on the clique (or independence) number.
There are many variations of such bounds known, including spectral \cite{Hoffman-1970-Eigenvalues,Haemers-1995-InterlacingEigenvaluesGraphs, GN-2008-EigenvalueBoundsIndependentSet,BT-2019-VectorColoringIrregular, Haemers-2021-HoffmanRatioBound}, LP \cite{Delsarte-1972-LPUnrestrictedCodes, Delsarte-1973-LPAssociationSchemes, Schrijver-1979-DelsarteLovaszComparison, SA-1990-Hierarchy}, and SDP \cite{Lovasz-1979-ShannonCapacityGraph,Knuth-1993-SandwichTheoremLovasz,Lovasz-2003-SDPCombinatorialOptimization,Laurent-2003-ComparisonRelaxations, Lasserre-2001-GlobalOptimizationMoments, Parrilo-2003-SDPSemialgebraic} bounds (we include references for bounds on the chromatic number as well, which are typically similar in spirit).
There is also another class of \emph{rank} or \emph{inertia} bounds, which instead only work with the number of positive, negative, and zero eigenvalues of suitable matrices, e.g., \cite{Haemers-1978-RankBoundShannonCapacity,GR-2001-AlgebraicGraphTheory,Sinkovic-2018-InertiaBoundNotTight}.
It is an interesting question to consider how these other bounds interact with localization.

\paragraph{Paley graph clique number bounds}
The history of upper bounds on $\omega(G_p)$ is brief: the $\sqrt{p}$ bound has long been folklore and admits an elementary combinatorial proof \cite[Lemma 1.2]{Yip-2022-CliqueNumberPaleyPrimePower}.
The first improvements \cite{MP-2006-ColoringPaley,BMR-2013-SquaresDifferenceSets} showed a bound of $\sqrt{p} - 1$ under different conditions on $p$.
The state-of-the-art result of \cite{HP-2021-SumsetsPaleyClique} was also the first to improve the constant in front of $\sqrt{p}$, using the polynomial method, and \cite{BSW-2021-ProductDirectionsGaloisPaley} gave an alternate proof soon after.
The lower bounds we mention in the introduction are number-theoretic improvements of the bound $\omega(G_p) \geq (\frac{1}{2} + o(1))\log_2 p$ which follows from the general upper bound on the Ramsey number $R(k)$ due to \cite{ES-1935-CombinatorialProblemGeometry} and the fact that $\omega(G_p) = \alpha(G_p)$.

\paragraph{Explicit semicircular elements}
Some recent work has pursued constructing explicit deterministic matrices whose e.s.d.\ is close (in some distributional distance such as Kolmogorov distance; see our Definition~\ref{def:kolmogorov}) to the semicircle law \cite{SXT-2017-PseudoWignerMatrices,SXT-2018-SymmetricPseudorandomMatrices}.
Since, after rescaling, the Kesten-McKay measures $\KM(v)$ approach the semicircle law as $v \to \infty$, our results suggest that any sequence of adjacency matrices $\bA_{G_{p, I}}$ or $\bS_{G_{p, I}}$ for $I \in \sI_a(G_p)$ with $a = a(p) \to \infty$ should be good candidates for such constructions, but we do not pursue a careful analysis here.

\paragraph{Deterministic asymptotic freeness}
We will see that our Conjecture~\ref{conj:weak-conv} follows from Conjecture~\ref{conj:asymp-free}, which posits that certain explicit deterministic pairs of projection matrices are asymptotically free.
We are not aware of many examples of deterministic asymptotically free pairs, particularly not for pairs of projections as in our case.
One such example is given in \cite[Example 9]{CVE-2012-PartialFreenessRandomMatrices}, and another for deterministic matrices with e.s.d.\ converging to the semicircle law as discussed above in \cite{ST-2018-AsymptoticallyIndependentPseudoWigner}.
Generally, it is an interesting problem to derandomize classical results of free probability, finding deterministic pairs of matrices under various restrictions that are asymptotically free.

\paragraph{Character sum estimates} Though necklace character sums can be written in the standard form $\sum_{\bx \in \FF_p^k} \chi(q(\bx))$ for $q$ a polynomial, we are not aware of any general purpose character sum estimates that treat our particular $q$ (see \cite{IK-2021-AnalyticNumberTheory} for a broad survey of the area; some results that handle situations similar to ours include \cite{Katz-1999-EstimatesSingularExponentialSums,RojasLeon-2005-SingularMultiplicativeCharacterSums,Katz-2007-NonsingularMixedCharacterSums,Katz-2008-MixedCharacterSumEstimates}).
The issue is that $q(\bx)$ is highly ``singular,'' in the sense that the variety $\{\bx: q(\bx) = 0\} \subseteq \FF_p^k$ is highly singular, since $q$ is a product of linear factors.
A very similar situation is treated in the recent work \cite{RojasLeon-2022-GeneralizationJacobiSums}, but unfortunately its assumptions still ask for the linear factors of $q$ to be in a ``general position'' that our setting does not satisfy.
We also point out the reference \cite{Chung-1989-DiametersEigenvalues}, which uses simpler character sum estimates to control the spectrum of a different pseudorandom graph construction.

\section{Notation}

We write $\bm I_n \in \RR^{n \times n}$ for the identity matrix of corresponding dimension.
We write $\one \in \RR^{\FF_p}$ for the all-ones vector, and $\what{\one} \colonequals \frac{1}{\sqrt{p}}\one$, so that $\|\what{\one}\| = 1$.
We write $\be_1, \dots, \be_n \in \RR^{n}$ for the standard orthonormal basis, with $(\be_i)_j = \One\{i = j\}$.
For $I \subseteq [n]$, we write $\bE_I = \sum_{i \in I} \be_i\be_i^{\top}$.
Matrix multiplications bind before the trace, so $\Tr \bA\bB = \Tr(\bA\bB)$ and $\Tr (\bA\bB)^k = \Tr((\bA\bB)^k)$.
For $\bA \in \RR^{n \times n}_{\sym}$, we write $\lambda_1(\bA) \geq \cdots \geq \lambda_n(\bA)$ for the ordered (real) eigenvalues.
We also write $\lambda_{\max}(\bA) \colonequals \lambda_1(\bA)$ and $\lambda_{\min}(\bA) \colonequals \lambda_n(\bA)$.
For $\bA \in \RR^{m \times n}$, we write $\|\bA\|$ for the operator norm or largest singular value, $\|\bA\|_F = \sqrt{\Tr(\bA\bA^{\top})}$ for the Frobenius norm, and $\|\bA\|_* = \Tr(\sqrt{\bA\bA^{\top}})$ for the nuclear norm.

For a graph $G$, we write $V(G)$ for the vertex set, $\deg(v)$ for the degree of $v \in V(G)$, and $\min\deg(G)$ and $\max\deg(G)$ for the minimum and maximum degree of any vertex in $G$, respectively.
We write $\bA_G$ for the $\{0, 1\}$-valued adjacency matrix, and $\bS_G = 2\bA_G - \one\one^{\top} + \bm I$ for the $\{\pm 1\}$-valued adjacency matrix.

All asymptotic notations $O(\cdot), o(\cdot), \omega(\cdot), \Omega(\cdot), \Theta(\cdot), \ll, \gg$ refer to the limit of $p \to \infty$ over primes $p \equiv 1 \Mod{4}$.
Subscripts such as $O_a(\cdot)$ denote which parameters the implicit constant depends on.

We reserve $\chi$ for the Legendre symbol, $\eps$ for the trivial character, and $\phi, \psi$ for general multiplicative characters of $\FF_p^{\times}$.
We write $\what{\FF_p^{\times}}$ for the \emph{dual group}, the set of all multiplicative characters.
We write $\overline{\phi}$ for the conjugate character $\overline{\phi}(x) = \overline{\phi(x)}$, and write $\phi\psi$ for pointwise multiplication, so that $\phi\psi(x) = \phi(x)\psi(x)$.

\section{Preliminaries}

\subsection{Paley graphs}

We recall a few basic and well-known properties of the Paley graph.

\begin{proposition}[Automorphisms]
    \label{prop:paley-aut}
    Let $\sigma_{a, b}: \FF_p \to \FF_p$ be given by $\sigma_{a, b}(x) = ax + b$.
    Then, $\{\sigma_{a, b}\}_{a \in \SS_p, b \in \FF_p}$ is the automorphism group of $G_p$, having size $\frac{p(p - 1)}{2}$.
\end{proposition}

\begin{corollary}[Transitivity]
    \label{cor:transitivity}
    There exists an automorphism of $G_p$ mapping any vertex to any other vertex (i.e., $G_p$ is vertex transitive), mapping any edge to any edge (i.e., $G_p$ is edge transitive), and any non-edge to any non-edge (i.e., $\overline{G_p}$ is edge transitive).
\end{corollary}

\begin{proposition}[Spectral decomposition]
    \label{prop:paley-spectral}
    Define matrices $\bP^{(\pm)}_p \in \RR^{\FF_p \times \FF_p}_{\sym}$ to have entries
    \begin{align}
      (\bP^{(+)}_p)_{ij} &\colonequals \left\{\begin{array}{rl} \frac{p - 1}{2p} & \text{if } i = j, \\
                                 \frac{\sqrt{p} - 1}{2p} & \text{if } i \neq j \text{ and } i \sim_{G_p} j, \\
                                 -\frac{\sqrt{p} + 1}{2p} & \text{if } i \neq j \text{ and } i \not \sim_{G_p} j, \end{array} \right. \\
      (\bP^{(-)}_p)_{ij} &\colonequals \left\{\begin{array}{rl} \frac{p - 1}{2p} & \text{if } i = j, \\
                                 -\frac{\sqrt{p} + 1}{2p} & \text{if } i \neq j \text{ and } i \sim_{G_p} j, \\
                                 \frac{\sqrt{p} - 1}{2p} & \text{if } i \neq j \text{ and } i \not \sim_{G_p} j. \end{array} \right.
    \end{align}
    Then, $\bP^{(\pm)}_p$ are orthogonal projections, each has rank $(p - 1) / 2$, and these matrices satisfy
    \begin{equation}
        \what{\one}\what{\one}^{\top} + \bP^{(+)}_p + \bP^{(-)}_p = \bm I_{\FF_p}.
    \end{equation}
    The $\{0, 1\}$ and $\{\pm 1\}$ adjacency matrices of $G_p$ then admit the following spectral decompositions:
    \begin{align*}
      \bA_{G_p}
      &= \frac{p - 1}{2}\what{\one}\, \what{\one}^{\top} + \frac{\sqrt{p} - 1}{2}\bP^{(+)}_p - \frac{\sqrt{p} + 1}{2} \bP^{(-)}_p, \numberthis \\
      \bS_{G_p}
      &= 2\bA_{G_p} - \one\one^{\top} + \bm I_{\FF_p} \\
      &= \sqrt{p}\, \bP_p^{(+)} - \sqrt{p}\, \bP^{(-)}_p. \numberthis
    \end{align*}
\end{proposition}

\subsection{Number theory}

We next gather several number-theoretic facts that we will need.
We adopt the standard notation
\begin{equation}
    e_p(x) \colonequals \exp\left(\frac{2\pi i}{p}x\right).
\end{equation}

\begin{definition}[Multiplicative characters]
    We write $\what{\FF_p^{\times}}$ for the group of multiplicative characters on $\FF_p$.
    Recall that $|\what{\FF_p^{\times}}| = p - 1$, and this group is isomorphic to the cyclic group of order $p - 1$.
    We write $\chi \in \what{\FF_p^{\times}}$ for the Legendre symbol and $\eps \in \what{\FF_p^{\times}}$ for the trivial character.
    We extend characters $\psi$ to have domain $\FF_p$ (including zero) by setting $\psi(0) \colonequals 0$ for all $\psi$ (including setting $\eps(0) \colonequals 0$, which is a common but not universal convention).
\end{definition}

\begin{definition}[Character sums]
The \emph{Gauss, Jacobi,} and \emph{Kloosterman} sums are defined as, respectively,
\begin{align}
    G(\psi) &\colonequals \sum_{x \in \FF_p} \psi(x) e_p(x), \\
    J(\psi, \phi) &\colonequals \sum_{x \in \FF_p} \psi(x)\phi(1 - x), \\
    K_k(a) &\colonequals \sum_{\substack{x_1, \dots, x_k \in \FF_p \\ x_1 \cdots x_k = a}} e_p(x_1 + \cdots + x_k),
\end{align}
for $\psi, \phi \in \what{\FF_p^{\times}}$ and $a \in \FF_p^{\times}$.
\end{definition}

The following results may be viewed as Fourier transforms of various functions on either the additive group $\FF_p$ or the multiplicative group $\FF_p^{\times}$.
There are many interesting identities of this kind connecting various exponential and character sums; the reader may consult, for instance, \cite[p.\ 47]{Katz-1988-GaussKloostermanMonodromy}.
\begin{proposition}
    \label{prop:char-fourier}
    For all $x\in \FF_p$,
    \begin{align}
        \chi(x) &= \frac{1}{\sqrt{p}} \sum_{a \in \FF_p} \chi(a) e_p(ax), \\
        \One\{x = 1\} &= \frac{1}{p - 1} \sum_{\psi \in \what{\FF_p^{\times}}} \psi(x).
    \end{align}
\end{proposition}

\paragraph{Character sum estimates}
The following is the main powerful theorem controlling univariate character sums.
\begin{theorem}[Weil's bound, Chapter 11 of \cite{IK-2021-AnalyticNumberTheory}]
    \label{thm:weil}
    Let $q \in \mathbb{F}_p[t]$ have $\deg q = d$. Suppose that $q$ cannot be represented as $q(t) = c \cdot r(t)^2$ for $c \in \mathbb{F}_p$ and $r \in \mathbb{F}_p[t]$. Then,
    \begin{equation}
      \label{eq:weil}
        \left| \sum_{x \in \mathbb{F}_p} \chi(q(x)) \right| \le (d - 1)\sqrt{p}.
    \end{equation}
\end{theorem}
\noindent
We refer to a sum as on the left-hand side of \eqref{eq:weil} as a \emph{Weil sum}.

It turns out to be possible to evaluate the magnitude of any Gauss sum, as follows.
\begin{proposition}[{\cite[Sections 6.4 and 8.2]{IR-1990-ClassicalModernNumberTheory}}]
    $G(\eps) = -1$, and $|G(\psi)| = \sqrt{p}$ for all $\psi \neq \eps$.
    Moreover, $G(\chi) = \sqrt{p}$ for all $p \equiv 1 \Mod{4}$.
\end{proposition}
\noindent
Because of this uniformity of the magnitude, other interesting properties of Gauss sums are often stated in terms of the complex modulus or ``angle,''
\begin{equation}
    g(\psi) \colonequals \frac{G(\psi)}{\sqrt{p}} \text{ for } \psi \neq \eps.
\end{equation}
The main intuition about these quantities is that they behave, when taken over all $\psi$, as though randomly distributed on the complex unit circle.
The following deep result of Katz quantifies this pseudorandomness or ``equidistribution.''
\begin{theorem}[{\cite[Theorem 9.6]{Katz-1988-GaussKloostermanMonodromy}}]
    \label{thm:katz-monomial-bound}
    For any $\phi_1, \dots, \phi_r \in \what{\FF_p^{\times}}$ distinct and $m_1, \dots, m_r \in \ZZ$,
    \begin{equation}
        \left|\sum_{\psi \in \what{\FF_p^{\times}} \setminus \{\overline{\phi_1}, \dots, \overline{\phi_r}\}} g(\phi_{1}\psi)^{m_1} \cdots g(\phi_{r}\psi)^{m_r}\right| \leq \left(\sum_{i = 1}^r |m_i|\right)\sqrt{p} + 2r.
    \end{equation}
\end{theorem}
\noindent
We emphasize the remarkable simplicity and non-asymptotic form of this result---this is thanks to the powerful algebraic-geometric machinery developed for character sum estimates in analytic number theory.
See \cite[Chapter 21]{IK-2021-AnalyticNumberTheory} for some further discussion of such equidistribution results.

It is also possible to reduce the problem of evaluating Jacobi sums to that of evaluating Gauss sums.
\begin{proposition}[{\cite[Section 8.3]{IR-1990-ClassicalModernNumberTheory}}]
    \label{prop:jacobi}
    Let $\phi, \psi \in \what{\FF_p^{\times}}$.
    Then,
    \begin{equation}
        J(\phi, \psi) = \left\{\begin{array}{ll}
        p - 2 & \text{if } \phi = \psi = \eps, \\ 0 & \text{if exactly one of } \phi, \psi \text{ is } \eps, \\
        -\phi(-1) & \text{if } \phi\psi = \eps \text{ and } \phi, \psi \neq \eps, \\ G(\phi) G(\psi) / G(\phi\psi) & \text{if } \phi, \psi, \phi\psi \neq \eps. \end{array}\right.
    \end{equation}
\end{proposition}
\noindent
We will use below that, through Katz's bound, this allows us to quantify cancellations in sums of polynomials of Jacobi sums (taken over multiplicative characters).

Similar results are also known for Kloosterman sums; the following result on ``moments'' of Kloosterman sums (in the same sense as above of a sum of monomials) that are ``twisted'' by a character is the one we will use.
For Kloosterman sums, even controlling the magnitude $|K_k(x)|$ is a deep result of Weil for $k = 2$ and of Deligne for general $k$ (see \cite{Deligne-1977-SGA4.5} or \cite[Chapter 11]{IK-2021-AnalyticNumberTheory}), who showed that
\begin{equation}
    |K_k(x)| \leq k p^{\frac{k - 1}{2}}.
\end{equation}
In light of this, the following should be viewed as describing square root cancellations in moments of Kloosterman sums.
\begin{theorem}[{\cite[Lemma 2.1]{LZZ-2018-DistributionJacobiSums}}]
    \label{thm:twisted-kloosterman}
    Let $\psi \in \what{\FF_p^{\times}}$ be non-trivial, $k \geq 1$, and $r, s \geq 0$.
    Then,
    \begin{equation}
        \left|\sum_{x \in \FF_p^{\times}} \psi(x) K_k(x)^r \overline{K_k(x)}^s\right| \leq k^{r + s - 1} p^{\frac{(k - 1)(r + s) + 1}{2}}.
    \end{equation}
\end{theorem}
\noindent
Similar results for the $k = 2$ case of Kloosterman sums are obtained in \cite{Liu-2002-KloostermanTwistedMoments} and \cite{CI-2000-CubicMomentLFunctions}.

\subsection{Weak convergence and asymptotic freeness}
\label{sec:prob}

The following is a standard folklore ``robustness'' statement for weak convergence.
We will use this to allow ourselves to ignore various corrections involved in translating between the $\{0, 1\}$ adjacency matrix of $G_p$, the $\{\pm 1\}$ adjacency matrix of $G_p$, and the projection matrices $\bP_p^{(\pm)}$ discussed above.
\begin{proposition}
    \label{prop:hoffman-wielandt}
    Suppose $\bA^{(N)}, \bB^{(N)} \in \RR^{N \times N}_{\sym}$ for $N = N(n)$ an increasing sequence are matrices such that $\frac{1}{N} \sum_{j = 1}^{N} \delta_{\lambda_j(\bA^{(N)})}$ converges weakly to a probability measure $\rho$ supported on $[-C, C] \subset \RR$ for some $C > 0$ as $n \to \infty$ and $\|\bA^{(N)} - \bB^{(N)}\|_* / N \to 0$ as $n \to \infty$.
    Then, $\frac{1}{N} \sum_{j = 1}^{N} \delta_{\lambda_j(\bB^{(N)})}$ converges weakly to $\rho$ as well.
\end{proposition}
\begin{proof}
    We will use the $\ell^1$ version of the Hoffman-Wielandt inequality \cite[Theorem III.4.4]{Bhatia-2013-MatrixAnalysis}, which gives
    \begin{equation}
        \lim_{N \to \infty}\frac{1}{N}\sum_{j = 1}^{N} | \lambda_j(\bA^{(N)}) - \lambda_j(\bB^{(N)}) | \leq \lim_{N \to \infty} \frac{1}{N}\|\bA^{(N)} - \bB^{(N)}\|_* = 0.
    \end{equation}
    By our assumption on $\rho$, we have
    \begin{equation}
        \lim_{N \to \infty} \frac{\#\{j: \lambda_j(\bA^{(N)}) \notin [-2C, 2C]\}}{N} = 0,
    \end{equation}
    and thus
    \begin{equation}
        \lim_{N \to \infty} \frac{\#\{j: \lambda_j(\bB^{(N)}) \notin [-3C, 3C]\}}{N} = 0.
    \end{equation}

    Now, let $h: \RR \to \RR$ be a bounded smooth function, and let $L \colonequals \max_{x \in [-3C, 3C]}|h^{\prime}(x)| < \infty$ and $M \colonequals \sup_{x \in \RR} |h(x)|$.
    Then,
    \begin{align*}
      &\hspace{-0.5cm}\left| \frac{1}{N}\sum_{j = 1}^{N} h(\lambda_j(\bA^{(N)})) - \frac{1}{N}\sum_{j = 1}^{N} h(\lambda_j(\bB^{(N)})) \right| \\
      &\leq \frac{1}{N}\sum_{j = 1}^{N} | h(\lambda_j(\bA^{(N)})) - h(\lambda_j(\bB^{(N)})) | \\
      &\leq \frac{L}{N}\sum_{j = 1}^{N} | \lambda_j(\bA^{(N)}) - \lambda_j(\bB^{(N)}) | \\
      &\hspace{1cm} + \frac{M}{N}\big(\#\{j: \lambda_j(\bA^{(N)}) \notin [-3C, 3C]\} + \#\{j: \lambda_j(\bB^{(N)}) \notin [-3C, 3C]\}\big). \numberthis
    \end{align*}
    Both terms on the right-hand side tend to zero by our observations above.
    The result follows by a standard approximation argument since this holds for any smooth bounded $h$.
\end{proof}

Next, let us introduce basic facts about asymptotic freeness of projection matrices.
For the sake of brevity, we omit standard background on free probability.
The reader may consult either the discussion in \cite{Kunisky-2023-MANOVAProductProjections} or references such as \cite{VDN-1992-FreeRandomVariables, NS-2006-LecturesCombinatoricsFreeProbability, MS-2017-FreeProbabilityRandomMatrices}.
For our purposes, the term ``asymptotically free'' may be taken as an abstract placeholder, since we will only use the sufficient condition below in our arguments.
\begin{proposition}[{\cite[Theorem 1.5]{Kunisky-2023-MANOVAProductProjections}}]
    \label{prop:proj-asymp-free}
    Suppose that $\bP_1^{(N)}, \bP_2^{(N)} \in \RR^{N \times N}_{\sym}$ for $N = N(n)$ an increasing sequence are orthogonal projection matrices satisfying the following properties:
    \begin{align}
      \lim_{n \to \infty} \frac{1}{N} \Tr(\bP_1^{(N)}) &= \alpha \in (0, 1), \label{eq:lim-alpha} \\
      \lim_{n \to \infty} \frac{1}{N} \Tr(\bP_2^{(N)}) &= \beta \in (0, 1), \label{eq:lim-beta} \\
      \lim_{n \to \infty} \frac{1}{N} \Tr\big((\bP_1^{(N)} - \alpha \bm I_N)(\bP_1^{(N)} - \beta \bm I_N)\big)^k &= 0 \text{ for all } k \geq 1.
    \end{align}
    Then, the pairs $(\bP_1^{(N)}, \bP_2^{(N)})$ are asymptotically free.
\end{proposition}

The purpose of proving asymptotic freeness is that, if it holds, then the limiting e.s.d.\ of $\bP_2^{(N)}\bP_1^{(N)}\bP_2^{(N)}$ is automatically determined (and is a \emph{multiplicative free convolution} of the limiting e.s.d.'s of $\bP_1^{(N)}$ and $\bP_2^{(N)}$, which are Bernoulli distributions $\Ber(\alpha)$ and $\Ber(\beta)$, respectively).
That limit belongs to the following family.
\begin{definition}[Wachter's MANOVA law \cite{Wachter-1980-EmpiricalMeasureDiscriminantRatios}]
    \label{def:manova}
    For two parameters $\alpha, \beta \in (0, 1)$, we say a random variable has the law $\MANOVA(\alpha, \beta) = \MANOVA(\beta, \alpha)$ if it has the following density with respect to Lebesgue measure:
    \begin{equation}
        \One\{x \in [r_-, r_+]\}\frac{\sqrt{(r_+ - x)(x - r_-)}}{2\pi x(1 - x)} \, dx + (1 - \min\{\alpha, \beta\})\delta_0(x) + \max\{\alpha + \beta - 1, 0\}\delta_1(x),
    \end{equation}
    where
    \begin{align*}
      r_{\pm} = r_{\pm}(\alpha, \beta) &\colonequals \alpha + \beta - 2\alpha\beta \pm 2\sqrt{\alpha(1 - \alpha)\beta(1 - \beta)} \\
              &= \left(\sqrt{\alpha(1 - \beta)} \pm \sqrt{\beta(1 - \alpha)}\right)^2.
    \end{align*}
\end{definition}

The following states the multiplicative free convolution result; the conclusion of weak convergence requires a small additional argument, which may for example be derived as a special case of \cite[Theorem 1.7]{Kunisky-2023-MANOVAProductProjections}.
\begin{proposition}
    \label{prop:proj-weak-conv}
    If a sequence of deterministic pairs of projection matrices $(\bP_1^{(N)}, \bP_2^{(N)})$ are asymptotically free and satisfy the limits \eqref{eq:lim-alpha} and \eqref{eq:lim-beta} for some $\alpha, \beta \in (0, 1)$, then the e.s.d.\ of $\bP_2^{(N)}\bP_1^{(N)}\bP_2^{(N)}$ converges weakly to $\MANOVA(\alpha, \beta)$ as $N \to \infty$.
\end{proposition}

Finally, towards using these results for our purposes, let us recall the relationship between the MANOVA laws and the Kesten-McKay laws.
This follows by a direct calculation from the definitions.

\begin{proposition}
    \label{prop:manova-km}
    Suppose $X \sim \MANOVA(\frac{1}{2}, \beta)$ with $\beta \leq \frac{1}{2}$.
    Then, $\frac{2}{\beta}(X - \frac{1}{2})$ conditional on $X \neq 0$ has the law $\KM(1 / \beta)$.
\end{proposition}

\section{Proof of Theorem~\ref{thm:conj-equiv}}
\label{sec:prf:thm:conj-equiv}

To show that Conjecture~\ref{conj:char-sum} implies Conjecture~\ref{conj:weak-conv}, we will introduce the following intermediate conjecture.
For the sake of brevity, let us write $\bP_p \colonequals \bP^{(+)}_p$ below.

\begin{conjecture}[Asymptotic freeness, degree $a$]
    \label{conj:asymp-free}
    For all $p$ sufficiently large, let $I \in \sI_a(G_p)$.
    Let $\bE_{V(G_{p, I})} \in \RR^{\FF_p \times \FF_p}$ be the diagonal matrix with $(\bE_{V(G_{p, I})})_{xx} \colonequals \One\{x \in V(G_{p, I})\}$.
    Then, the sequence of pairs of (deterministic) matrices $(\bP_{p}, \bE_{V(G_{p, I})})$ are asymptotically free as $p \to \infty$.
\end{conjecture}

\begin{lemma}
    \label{lem:asymp-free-weak-conv}
    If Conjecture~\ref{conj:asymp-free} holds at degree $a$, then Conjecture~\ref{conj:weak-conv} holds at degree $a$.
\end{lemma}

\noindent
In proving this, we will need to control the rank of the $\bE_{V(G_{p, I})}$, which is $\Tr(\bE_{V(G_{p, I})}) = |V(G_{p, I})|$.
This is achieved by the following character sum computation.
\begin{proposition}[Size of localization]
    \label{prop:loc-graph-vertices}
    For any $I \in \sI_a(G_p)$,
    \begin{equation}
        \left||V(G_{p, I})| - \frac{p}{2^a}\right| \leq (a - 1)\sqrt{p} + \frac{a}{2}
    \end{equation}
\end{proposition}
\begin{proof}
    We have
    \begin{align*}
      |V(G_{p, I})|
      &= -\frac{a}{2} + \sum_{x \in \FF_p} \prod_{y \in I} \left(\frac{1 - \chi(x - y)}{2}\right) \\
      &= \frac{p}{2^a} - \frac{a}{2} + \frac{1}{2^a} \sum_{\substack{A \subseteq I \\ A \neq \emptyset}}(-1)^{|A|}\left(\sum_{x \in \FF_p} \prod_{y \in A} \chi(x - y)\right). \numberthis
    \end{align*}
    The result then follows since there are $2^k - 1$ terms in the last sum, and each one is a Weil sum bounded in magnitude by $(a - 1)\sqrt{p}$ by Theorem~\ref{thm:weil}.
\end{proof}

\begin{proof}[Proof of Lemma~\ref{lem:asymp-free-weak-conv}]
    We have
    \begin{align}
      \Tr(\bP_p) &= \frac{p - 1}{2}, \tag{by Proposition~\ref{prop:paley-spectral}} \\
      \left|\Tr(\bE_{V(G_{p, I})}) - \frac{p}{2^a}\right| &= o(p). \tag{by Proposition~\ref{prop:loc-graph-vertices}}
    \end{align}
    Thus, Proposition~\ref{prop:proj-weak-conv} implies that, if Conjecture~\ref{conj:asymp-free} holds at degree $a$, then the e.s.d.'s of $\bE_{V(G_{p, I})} \bP_p \bE_{V(G_{p, I})}$ converge weakly to $\MANOVA(\frac{1}{2}, \frac{1}{2^a})$.

    Let $\bV \in \RR^{V(G_{p, I}) \times \FF_p}$ be the truncation of $\bE_{V(G_{p, I})}$ formed by removing the rows that are identically zero, so that $\bE_{V(G_{p, I})} = \bV^{\top}\bV$ and $\bA_{G_{p, I}} = \bV \bA_{G_p} \bV^{\top}$.
    By Propositions~\ref{prop:loc-graph-vertices} and~\ref{prop:manova-km}, the e.s.d.\ of $2^{a + 1}\bV(\bP_p - \frac{1}{2}\bm I_p)\bV^{\top}$ converges weakly to $\KM(2^a)$.
    We note that replacing $\bE_{V(G_{p, I})}$ by $\bV$ removes $p - \rank(\bE_{V(G_{p, I})}) = (1 - \frac{1}{2^a} + o(1))p$ zero eigenvalues, which has the same effect as the conditioning of $X$ in Proposition~\ref{prop:manova-km}.

    Now, by rearranging in Proposition~\ref{prop:paley-spectral}, we have
    \begin{equation}
        \frac{1}{\sqrt{p}}\bA_{G_p} = \bP^{(+)} - \frac{1}{2}\bm I_p - \frac{1}{2\sqrt{p}}\bm I + \frac{\sqrt{p} + 1}{2}\what{\one}\, \what{\one}^{\top}.
    \end{equation}
    In particular, we have
    \begin{equation}
        \left\| \frac{1}{\sqrt{p}}\bA_{G_p} - \left(\bP^{(+)} - \frac{1}{2}\bm I_p\right) \right\|_* = O(\sqrt{p}) = o(p),
    \end{equation}
    and so we also have
    \begin{equation}
        \bigg\| \frac{2^{a + 1}}{\sqrt{p}}\underbrace{\bV\bA_{G_p}\bV^{\top}}_{= \bA_{G_{p, I}}} - 2^{a + 1}\bV\left(\bP^{(+)} - \frac{1}{2}\bm I_p\right)\bV^{\top} \bigg\|_* = O(\sqrt{p}) = o(p),
    \end{equation}
    and the result follows by Proposition~\ref{prop:hoffman-wielandt}.
\end{proof}

\begin{proof}[Proof of Theorem~\ref{thm:conj-equiv}]
    By Lemma~\ref{lem:asymp-free-weak-conv}, it suffices to show under the assumptions of the Theorem that the sequence of pairs $(\bP_p, \bE_{V(G_{p, I})})$ are asymptotically free.
    By Proposition~\ref{prop:proj-asymp-free}, it suffices in turn to show that, for each $k \geq 1$,
    \begin{equation}
        \lim_{p \to \infty} \frac{1}{p} \underbrace{\Tr\left(\left(\bP_p - \frac{1}{2}\bm I\right)\left(\bE_{V(G_{p, I})} - \frac{1}{2^a}\bm I\right)\right)^k}_{\equalscolon T_k} = 0,
    \end{equation}
    where we substitute the values of $\alpha$ and $\beta$ from the proof of Lemma~\ref{lem:asymp-free-weak-conv} above.

Let us write
\begin{align}
  \wtil{\bE} &\colonequals \bE_{V(G_{p, I})} - \frac{1}{2^a}\bm I_{\FF_p}, \\
  \wtil{\bP} &\colonequals \bP_p - \frac{1}{2}\bm I_{\FF_p}, \\
  \bS &\colonequals \bS_{G_p}.
\end{align}
Rearranging in Proposition~\ref{prop:paley-spectral}, we have
\begin{equation}
    \wtil{\bP} = \frac{1}{2}\left(\frac{1}{\sqrt{p}}\bS - \what{\one}\what{\one}^{\top}\right).
\end{equation}

Substituting and expanding, we find
\begin{equation}
  T_k = \frac{1}{2^k}\sum_{s_1, \dots, s_k \in \{0, 1\}} \Tr \bH_{s_1}^{(0)}\wtil{\bE} \cdots \bH_{s_k}^{(0)}\wtil{\bE},
\end{equation}
where
\begin{align}
  \bH_0^{(0)} &\colonequals \frac{1}{\sqrt{p}} \bS \equalscolon \what{\bS}, \\
  \bH_1^{(0)} &\colonequals \what{\one}\what{\one}^{\top}.
\end{align}
Note that all three matrices $\bH_0^{(0)} = \what{\bS}, \bH_1^{(0)},$ and $\wtil{\bE}$ have operator norm at most 1.
For any term in the sum with some $s_i = 1$, we may then bound
\begin{align*}
  \left|\Tr \bH_{s_1}^{(0)}\wtil{\bE} \cdots \bH_{s_k}^{(0)}\wtil{\bE}\right|
  &\leq \left|\what{\one}^{\top} \left(\wtil{\bE} \bH_{s_{i + 1}}^{(0)} \cdots \wtil{\bE} \bH_{s_{k}}^{(0)} \wtil{\bE} \bH_{s_{1}}^{(0)} \cdots \wtil{\bE} \bH_{s_{i - 1}}^{(0)}\right) \what{\one}\right| \\
  &\leq \|\wtil{\bE} \bH_{s_{i + 1}}^{(0)} \cdots \wtil{\bE} \bH_{s_{k}}^{(0)} \wtil{\bE} \bH_{s_{1}}^{(0)} \cdots \wtil{\bE} \bH_{s_{i - 1}}^{(0)}\| \\
  &\leq 1. \numberthis
\end{align*}
Thus we find
\begin{equation}
    |T_k| \leq 1 + \frac{1}{2^k}\big|\underbrace{\Tr(\what{\bS} \wtil{\bE})^k}_{\equalscolon T_k^{(1)}}\big|.
\end{equation}

Now let us make a similar adjustment to the matrix $\wtil{\bE}$.
We may express, for all $x \in \FF_p$,
\begin{equation}
    (\bE_{V(G_{p, I})})_{xx} = \prod_{z \in I} \frac{1 - \chi(x - z)}{2} - \frac{1}{2}\One\{x \in I\},
\end{equation}
and thus
\begin{equation}
    \wtil{E}_{xx} = \left(\bE_{V(G_{p, I})} - \frac{1}{2^a}\bm I\right)_{xx} = \frac{1}{2^a}\sum_{\substack{A \subseteq I \\ A \neq \emptyset}}\left(-1\right)^{|A|} \prod_{z \in A} \chi(x - z) - \frac{1}{2}\One\{x \in I\}.
\end{equation}
Let us write $\bD \in \RR^{\FF_p \times \FF_p}$ for the diagonal matrix with entries
\begin{equation}
    D_{xx} \colonequals \frac{1}{2^a}\sum_{\substack{A \subseteq I \\ A \neq \emptyset}}\left(-1\right)^{|A|} \prod_{z \in A} \chi(x - z) = \left\{\begin{array}{rl} \frac{1}{2} - \frac{1}{2^a} & \text{if } x \in I, \\ 1 - \frac{1}{2^a} & \text{if } x \in V(G_{p, I}), \\ -\frac{1}{2^a} & \text{otherwise.} \end{array}\right.
\end{equation}
We then have
\begin{equation}
    \wtil{\bE} = \bD - \frac{1}{2}\bE_I.
\end{equation}
Substituting this into the expression for $T_k^{(1)}$, we have
\begin{equation}
  T_k^{(1)}
  = \sum_{s_1, \dots, s_k \in \{0, 1\}} \Tr \bH_{s_1}^{(1)}\what{\bS} \cdots \bH_{s_k}^{(1)}\what{\bS},
\end{equation}
where
\begin{align}
  \bH_0^{(1)} &\colonequals \bD, \\
  \bH_1^{(1)} &\colonequals -\frac{1}{2}\bE_I.
\end{align}
Note again that all three matrices $\bH_0^{(1)}, \bH_1^{(1)}$, and $\what{\bS}$ have operator norm at most 1.
For any term in the sum with some $s_i = 1$, we may use that $\bE_I = \sum_{x \in I} \be_x\be_x^{\top}$ to bound
\begin{align*}
  \left|\Tr \bH_{s_1}^{(1)}\what{\bS} \cdots \bH_{s_k}^{(1)} \what{\bS}\right|
  &\leq \sum_{x \in I} \left|\be_x^{\top} \left(\what{\bS} \bH_{s_{i + 1}}^{(1)} \cdots \what{\bS} \bH_{s_{k}}^{(1)} \what{\bS} \bH_{s_{1}}^{(1)} \cdots \what{\bS} \bH_{s_{i - 1}}^{(1)}\right) \be_x\right| \\
  &\leq a\|\what{\bS} \bH_{s_{i + 1}}^{(1)} \cdots \what{\bS} \bH_{s_{k}}^{(1)} \what{\bS} \bH_{s_{1}}^{(1)} \cdots \what{\bS} \bH_{s_{i - 1}}^{(1)}\| \\
  &\leq a. \numberthis
\end{align*}
Thus, we find
\begin{equation}
    |T_k^{(1)}| \leq a2^k + \big| \underbrace{\Tr(\bD \what{\bS})^k}_{\equalscolon T_k^{(2)}} \big|.
\end{equation}

Finally, we may apply our assumed character sum estimate to $T_k^{(2)}$.
For $Z \subseteq \FF_p$, let us write $\bD_Z$ for the diagonal matrix with entries
\begin{equation}
    (\bD_Z)_{xx} = \prod_{z \in Z} \chi(x - z).
\end{equation}
Then, the above says
\begin{equation}
    \bD = \frac{1}{2^a} \sum_{\substack{Z \subseteq I \\ Z \neq \emptyset}} (-1)^{|Z|} \bD_Z.
\end{equation}
Expanding this sum in $T_k^{(2)}$,
\begin{equation}
    |T_k^{(2)}| \leq \frac{1}{2^{ak}p^{k/2}} \sum_{\substack{Z_1, \dots, Z_k \subseteq I \\ Z_i \neq \emptyset}} |\Tr(\bD_{Z_1} \bS \cdots \bD_{Z_k} \bS)|.
\end{equation}
Expanding each trace, we find that each term is precisely a necklace character sum,
\begin{equation}
    \Tr(\bD_{Z_1} \bS \cdots \bD_{Z_k} \bS) = \Sigma(Z_1, \dots, Z_k).
\end{equation}
And, we have $Z_1 \cup \cdots \cup Z_k \subseteq I$, so $|Z_1 \cup \cdots \cup Z_k| \leq a$.
Thus, by our assumption of Conjecture~\ref{conj:char-sum} at degree $a$, we find
\begin{equation}
    \lim_{p \to \infty} \frac{1}{p} |T_k^{(2)}| = 0,
\end{equation}
and substituting into the previous expressions for $T_k^{(1)}$ and then $T_k$ gives the result.
\end{proof}

\section{Proof of Theorem~\ref{thm:level1}}
\label{sec:pf:thm:level1}

The degree 1 setting of Conjecture~\ref{conj:char-sum} is in terms of non-empty subsets $Z_1, \dots, Z_k \subseteq \FF_p$ with $|Z_i| = 1$ and $|Z_1 \cup \cdots \cup Z_k| \leq 1$.
The only way that this can happen is if $Z_1 = \cdots = Z_k = \{z\}$ for some $z \in \FF_p$.
In this case, we have
\begin{align*}
  \Sigma
  &\colonequals \Sigma(\{z\}, \dots, \{z\}) \\
  &= \sum_{x_1, \dots, x_k \in \FF_p} \chi(x_2 - x_1) \cdots \chi(x_k - x_{k - 1}) \chi(x_1 - x_k) \prod_{i = 1}^k \chi(x_i - z) \\
  &= \sum_{x_1, \dots, x_k \in \FF_p} \chi(x_2 - x_1) \cdots \chi(x_k - x_{k - 1}) \chi(x_1 - x_k) \prod_{i = 1}^k \chi(x_i), \numberthis
\end{align*}
the last step following by changing variables $x_i \rightarrow x_i + z$, effectively allowing us to assume without loss of generality that $z = 0$.
We give two approaches, appealing to different prior work in number theory.
In this case these are actually equivalent by elementary manipulations (essentially because Kloosterman sums and products of Gauss sums are each other's Fourier transforms, as discussed in \cite[Chapter 4]{Katz-1988-GaussKloostermanMonodromy}), but we include both in case one treatment may prove easier to generalize to other sums.

\paragraph{Approach 1: Moments of Jacobi sums}
The specific bound we will show is
\begin{equation}
    |\Sigma| \leq 2k p^{(k + 1)/2} + 4p^{k/2} + 1.
\end{equation}
When we write expressions of the form $\chi(x^{-1})$ below, the inversion is meant to be taken in $\FF_p^{\times}$.
Since $\chi(0) = 0$, the terms in $\Sigma$ where any $x_i = 0$ will not contribute, so we have
\begin{align*}
  \Sigma
  &= \sum_{x_1, \dots, x_k \in \FF_p^{\times}} \chi(x_2 - x_1) \cdots \chi(x_k - x_{k - 1}) \chi(x_1 - x_k) \prod_{i = 1}^k \chi(x_i) \\
  &= \sum_{x_1, \dots, x_k \in \FF_p^{\times}} \chi(x_2 - x_1) \cdots \chi(x_k - x_{k - 1}) \chi(x_1 - x_k) \prod_{i = 1}^k \chi(x_i^{-1})
    \intertext{and using multiplicativity of $\chi$ to combine pairs of terms,}
  &= \sum_{x_1, \dots, x_k \in \FF_p^{\times}} \chi(1 - x_2x_1^{-1}) \cdots \chi(1 - x_kx_{k-1}^{-1}) \chi(1 - x_1x_k^{-1}) \\
  &= (p - 1) \sum_{\substack{x_1, \dots, x_k \in \FF_p \\ x_1 \cdots x_k = 1}} \chi(1 - x_1) \cdots \chi(1 - x_k)
  \intertext{Now, taking a Fourier transform of $\One\{x_1 \cdots x_k = 1\}$ using Proposition~\ref{prop:char-fourier},}
  &= \sum_{\psi \in \widehat{\FF_p^{\times}}}\sum_{x_1, \dots, x_k \in \FF_p} \chi(1 - x_1) \cdots \chi(1 - x_k) \psi(x_1 \cdots x_k) \\
  &= \sum_{\psi \in \widehat{\FF_p^{\times}}}\left(\sum_{x \in \FF_p} \chi(1 - x) \psi(x)\right)^k \\
  &= \sum_{\psi \in \widehat{\FF_p^{\times}}} J(\chi, \psi)^k \\
  &= (-1)^k + p^{k/2}\sum_{\psi \in \widehat{\FF_p^{\times}} \setminus \{\chi, \eps\}}g(\psi)^kg(\chi\psi)^{-k}, \numberthis
\end{align*}
where we finish by Proposition~\ref{prop:jacobi}.
Finally, the result now follows by an application of Katz's bound of Theorem~\ref{thm:katz-monomial-bound}.

\begin{remark}
    Similar results on the moments of Jacobi sums are proved in \cite{KZ-1996-DistributionGaussJacobiSums, Xi-2018-EquidistributionJacobiSums, LZZ-2018-DistributionJacobiSums}; however, they involve summing over pairs of characters $\sum_{\psi_1, \psi_2} J(\psi_1, \psi_2)^k$, whereas our situation requires us to fix $\psi_1 = \chi$.
\end{remark}

\paragraph{Approach 2: Twisted moments of Kloosterman sums}
Using an alternative method we can show the slightly tighter
\begin{equation}
    \label{eq:level1-explicit-bound}
    |\Sigma| \leq k p^{(k + 1)/2}.
\end{equation}
We start similarly to before, but now use the Gauss sum expansion from Proposition~\ref{prop:char-fourier}:
\begin{align*}
  \Sigma
  &= (p - 1) \sum_{\substack{x_1, \dots, x_k \in \FF_p \\ x_1 \cdots x_k = 1}} \chi(1 - x_1) \cdots \chi(1 - x_k) \\
  &= \frac{p - 1}{p^{k/2}}  \sum_{\substack{x_1, \dots, x_k \in \FF_p \\ x_1 \cdots x_k = 1}} \sum_{y_1, \dots, y_k \in \FF_p} \chi(y_1) \cdots \chi(y_k) e_p\big((1 - x_1)y_1 + \cdots + (1 - x_k)y_k\big) \\
  &= \frac{p - 1}{p^{k/2}} \sum_{y_1, \dots, y_k \in \FF_p} \chi(y_1 \cdots y_k) e_p(y_1 + \cdots + y_k) \overline{\sum_{\substack{x_1, \dots, x_k \in \FF_p \\ x_1 \cdots x_k = 1}} e_p\big(x_1y_1 + \cdots + x_ky_k\big)} \\
  &= \frac{p - 1}{p^{k/2}} \sum_{y_1, \dots, y_k \in \FF_p} \chi(y_1 \cdots y_k) e_p(y_1 + \cdots + y_k) \overline{K_k(y_1 \cdots y_k)} \\
  &= \frac{p - 1}{p^{k/2}} \sum_{b \in \FF_p^{\times}} \chi(b)\overline{K_k(b)} \sum_{\substack{y_1, \dots, y_k \in \FF_p \\ y_1 \cdots y_k = b}} e_p(y_1 + \cdots + y_k) \\
  &= \frac{p - 1}{p^{k/2}} \sum_{b \in \FF_p^{\times}} \chi(b)|K_k(b)|^2. \numberthis
\end{align*}
The result then follows from Theorem~\ref{thm:twisted-kloosterman}.

\paragraph{The minimum eigenvalue}
Lastly, we must prove that Conjecture~\ref{conj:min-eval} holds at degree $a = 1$.
Note that, for $a = 1$, the statement of the conjecture says that
\begin{equation}
    \lim_{p \to \infty} \frac{-\lambda_{\min}(\bA_{G_{p, \{z\}}})}{\sqrt{p}} = 1.
\end{equation}
The ``$\geq$'' direction of this claim follows from the weak convergence we have just proved.
For the ``$\leq$'' direction, we may just observe that $-\lambda_{\min}(\bA_{G_{p, \{z\}}}) \leq \|\bA_{G_{p, \{z\}}}\| \leq \|\bA_{G_p}\| = \sqrt{p} + 1$ by Proposition~\ref{prop:paley-spectral}.

\begin{remark}
    \label{rem:Gp0-edges}
    The above argument tells us that, while the \emph{shape} of the spectrum of $\bA_{G_{p, \{z\}}}$ is different from that of $\bA_{G_p}$, its (left) \emph{edge} remains the same.
    See Figure~\ref{fig:example-loc} for an illustration.
    Intuitively, the reason for this is that the bottom eigenspace of $\bA_{G_p}$ has dimension roughly $\frac{1}{2}p$, and the restriction to the induced subgraph $G_{p, \{z\}}$ may be viewed as conjugation of $\bA_{G_p}$ by a projection to a subspace of dimension roughly $\frac{1}{2}p$.
    Thus, these two subspaces are just on the threshold of needing to have a non-trivial intersection by dimension counting.
\end{remark}

\subsection{Equidistribution interpretation}
\label{sec:equidistribution-level1}

Let us discuss an alternative, more number-theoretic interpretation of our proof.
We have mentioned earlier that the $G_{p, \{z\}}$ are all isomorphic (see Corollary~\ref{cor:transitivity}), so let us fix $z = 0$ here for the sake of convenience.
Then, $V(G_{p, \{0\}}) = \FF_p^{\times} \setminus \SS_p$, the set of quadratic non-residues.

As used previously by \cite{MMP-2019-LPCliquesPaley}, the graph $G_{p, \{0\}}$ is circulant (meaning its adjacency matrix is a circulant matrix) when its vertices are ordered as $g, g^3, \dots, g^{p - 2}$ for $g$ a multiplicative generator of $\FF_p^{\times}$.
Under this ordering, its adjacency matrix is circulant with first row given by the sequence $\One\{g^{2j + 1} - g \in \SS_p\} = \frac{1}{2}(1 + \chi(g^{2j + 1} - g)) - \frac{1}{2}\One\{j = \frac{p - 1}{2}\} = \frac{1}{2}(1 - \chi(g^{2j} - 1)) - \frac{1}{2}\One\{j = \frac{p - 1}{2}\}$ for $j = 1, \dots, \frac{p - 1}{2}$.
The eigenvalues may be computed by taking the Fourier transform of this sequence,
\begin{equation}
    -\frac{1}{2} + \frac{1}{2}\sum_{j = 1}^{(p - 1) / 2} (1 - \chi(g^{2j} - 1)) e_{p - 1}(2aj) \text{ for } a = 1, \dots, \frac{p - 1}{2}.
\end{equation}
Each function $g^j \mapsto e_{p - 1}(aj)$ is a multiplicative character of $\FF_p^{\times}$, so we may equivalently express the eigenvalues as a small adjustment of the character sums
\begin{equation}
    S(\chi, \psi) \colonequals \sum_{y \in \FF_p} \chi(y^2 - 1) \psi(y^2) \text{ for } \psi \in \what{\FF_p^{\times}},
\end{equation}
with the redundancy that under this enumeration every eigenvalue will appear twice, since $\psi(y^2) = \chi\psi(y^2)$.
One intuitive description is that the $S(\chi, \psi)$ are Jacobi sums restricted to summing over quadratic residues.
In any case, the weak convergence part of Theorem~\ref{thm:level1} may be viewed as an \emph{equidistribution} result on the sums $S(\chi, \psi)$ as $\psi$ varies while $\chi$ is fixed to be the Legendre symbol.
The result says that the limiting empirical distribution of these character sums converges weakly (after rescaling) to $\KM(2)$, also known as the \emph{arcsine law}, having density
\begin{equation}
    \One\{|x| \leq 2\} \frac{1}{\pi \sqrt{4 - x^2}}\, dx.
\end{equation}

In this formulation, this equidistribution result is also proved in \cite{Xi-2021-DistributionMultiplicativeKloostermanSum}, where the $S(\chi, \psi)$ arise in relation to a variant of the Kloosterman sums.
We believe our proof of this equidistribution result is cleaner and more conceptual: we need only establish sufficiently cancellations in necklace character sums rather than also identifying ``main terms'' in empirical moments of the $S(\chi, \psi)$, and the interpretation of these character sums as eigenvalues of a suitable matrix gives a free probability explanation for why the arcsine law appears as the limiting distribution (these are similar to the advantages of the approach of \cite{Kunisky-2023-MANOVAProductProjections} over that of \cite{MMP-2019-RandomSubensembles} for analyzing the spectrum of the random subgraphs $H_{p, \beta}$).

The arcsine law has appeared at least once before in this context as the limiting law of twisted Kloosterman sums, but only when summing over a finite field $\FF_{p^t}$ with $t \geq 2$ \cite{Kelmer-2010-DistributionTwistedKloostermanPrimePower}.
In contrast, the same sums over $\FF_p$ have the semicircle law as a limiting empirical distribution \cite{Katz-1988-GaussKloostermanMonodromy}; generally, the semicircle law, known in number theory as the \emph{Sato-Tate law}, is believed to be the limiting empirical distribution of many character sums and related quantities \cite[Chapter~21]{IK-2021-AnalyticNumberTheory}.
Our argument relating an equidistribution result for character sums with free probability suggests some interesting questions for further investigation:
\begin{enumerate}
    \item Do Kesten-McKay laws with other parameters appear as the limiting empirical distributions of natural families of character sums?
    \item Do other equidistribution results of analytic number theory have interpretations in terms of free probability, and in particular in terms of asymptotic freeness of deterministic but pseudorandom matrices, as we find in this case?
\end{enumerate}

\section{Proof of Theorem~\ref{thm:clique-bound}}

To convert control of the minimum eigenvalue into control of the independence number, we will use the following general spectral bound on the independence number, an adaptation to arbitrary graphs of the well-known ``Hoffman bound'' for regular graphs.
\begin{proposition}
    \label{prop:haemers}
    For any graph $G$,
    \begin{equation}
        \alpha(G) \leq |V(G)| \left(\frac{\mindeg(G)^2}{-\lambda_{\min}(\bA_{G})\maxdeg(G)} + 1\right)^{-1}.
    \end{equation}
\end{proposition}
\noindent
The result follows by starting with Theorem 3.3 of \cite{Haemers-1995-InterlacingEigenvaluesGraphs}, which states the same inequality with $\maxdeg(G)$ replaced by $\lambda_{\max}(\bA_{G})$, and using that $\lambda_{\max}(\bA_{G}) \leq \maxdeg(G)$ by the Gershgorin circle theorem.

We need only collect bounds on the quantities other than $\lambda_{\min}(\bA_{G_{p, I}})$ appearing in Proposition~\ref{prop:haemers} (when it is applied to $G_{p, I}$), namely $|V(G_{p, I})|$, $\mindeg(G_{p, I})$, and $\maxdeg(G_{p, I})$.
Recall that Proposition~\ref{prop:loc-graph-vertices} controls $|V(G_{p, I})|$.

\begin{proposition}[Degrees in localizations]
    \label{prop:loc-graph-degrees}
    For any $I \in \sI_a(G_p)$,
    \begin{equation}
        \frac{p}{2^{a + 1}} - a\sqrt{p} - \frac{a + 1}{2} \leq \mindeg(G_{p, I}) \leq \maxdeg(G_{p, I}) \leq \frac{p}{2^{a + 1}} + a\sqrt{p} + \frac{a + 1}{2}
    \end{equation}
\end{proposition}
\begin{proof}
    Suppose that $x \in V(G_{p, I})$.
    Then, we note that by definition $I \cup \{x\} \in \sI_{a + 1}(G_p)$, and $\deg_{G_{p, I}}(x) = |V(G_{p, I \cup \{x\}})|$.
    The result then follows by Proposition~\ref{prop:loc-graph-vertices} applied to $I \cup \{x\}$.
\end{proof}

\begin{proof}[Proof of Theorem~\ref{thm:clique-bound}]
    We have by Propositions~\ref{prop:loc} and \ref{prop:haemers} that
    \begin{align*}
      &\omega(G_p) = \alpha(G_p) \\
      &\leq a + \max_{I \in \sI_a(G_p)} |V(G_{p, I})|\left(\frac{\mindeg(G_{p, I})^2}{-\lambda_{\min}(\bA_{G_{p, I}})\maxdeg(G_{p, I})} + 1\right)^{-1}
        \intertext{where the inner expression is monotonically increasing in $|V(G_{p, I})|$ and $\maxdeg(G_{p, I})$, and monotonically decreasing in $\mindeg(G_{p, I})$ and $\lambda_{\min}(\bA_{G_{p, I}}) < 0$. Thus, substituting from Propositions~\ref{prop:loc-graph-vertices} and \ref{prop:loc-graph-degrees},}
      &\leq a + \left(\frac{p}{2^a} + (a - 1)\sqrt{p} + \frac{a}{2}\right)\left(\frac{(\frac{p}{2^{a + 1}} - a\sqrt{p} - \frac{a + 1}{2})^2}{\max_{I \in \sI_a(G_p)}\{-\lambda_{\min}(\bA_{G_{p, I}})\} (\frac{p}{2^{a + 1}} + a\sqrt{p} + \frac{a + 1}{2})} + 1\right)^{-1}
        \intertext{and substituting the assumption of Conjecture~\ref{conj:min-eval} and estimating for $a$ fixed and as $p \to \infty$,}
      &\leq a + \left(\frac{p}{2^a} + O(\sqrt{p})\right)\left(\frac{(\frac{p}{2^{a + 1}} - O(\sqrt{p}))^2}{(\frac{\sqrt{2^a - 1}}{2^a}\sqrt{p} + o(\sqrt{p}))(\frac{p}{2^{a + 1}} + O(\sqrt{p}))} + 1\right)^{-1} \\
      &= \frac{\sqrt{2^a - 1}}{2^{a - 1}}\sqrt{p} + o(\sqrt{p}), \numberthis
    \end{align*}
    completing the proof.
\end{proof}

\section{Challenges in controlling the minimum eigenvalue}
\label{sec:challenges-eval}

Let us discuss the additional difficulty of proving Conjecture~\ref{conj:min-eval} as compared to Conjecture~\ref{conj:weak-conv}.
Recall that we showed that necklace character sum estimates imply Conjecture~\ref{conj:weak-conv} through the sufficient condition for asymptotic freeness of projection matrices (our Proposition~\ref{prop:proj-asymp-free}) from the companion work \cite{Kunisky-2023-MANOVAProductProjections}.
The same work also gives the following sufficient condition for convergence of an extreme eigenvalue, under a similar but stronger trace estimate.

\begin{theorem}[{\cite[Theorem 1.12]{Kunisky-2023-MANOVAProductProjections}}]
    \label{thm:edge-km}
    Suppose that $\bP_1^{(N)}, \bP_2^{(N)} \in \RR^{N \times N}_{\sym}$ for $N = N(n)$ an increasing sequence are orthogonal projection matrices that satisfy the assumptions of Proposition~\ref{prop:proj-asymp-free} with $\alpha = \frac{1}{2}$ and $\beta \in (0, 1)$.
    Suppose moreover that there is a sequence $k = k(n)$ such that $k \gg \log N$ and for which
    \begin{equation}
        \label{eq:max-eval-cond}
        \max_{1 \leq k^{\prime} \leq k} \left(\frac{2}{\sqrt{\beta(1 - \beta)}}\right)^{k^{\prime}}\left| \Tr\big((\bP^{(N)}_1 - \alpha \bm I_N)(\bP^{(N)}_2 - \beta \bm I_N)\big)^{k^{\prime}} \right| \leq \exp(o(k)).
    \end{equation}
    Then,
    \begin{equation}
        \label{eq:max-eval-lim}
        \lim_{n \to \infty}\lambda_{\max}(\bP_2^{(N)}\bP_1^{(N)}\bP_2^{(N)}) = \frac{1}{2} + \sqrt{\beta(1 - \beta)}.
    \end{equation}
\end{theorem}
\noindent
The right-hand side of \eqref{eq:max-eval-lim} is the right edge of the support of $\MANOVA(\frac{1}{2}, \beta)$, so it is not surprising that this value appears.

While it is not immediately obvious, let us show how, if it were possible to apply this Theorem in a similar setting to our use of Proposition~\ref{prop:proj-weak-conv}, then Conjecture~\ref{conj:min-eval} would follow.
We have
\begin{align*}
  \lambda_{\min}(\bA_{G_{p, I}})
  &= \lambda_{\min}(\bE_{V(G_{p, I})} \bA_{G_p} \bE_{V(G_{p, I})})
    \intertext{and, substituting a rearrangement of Proposition~\ref{prop:paley-spectral},}
  &= \lambda_{\min}\left(\bE_{V(G_{p, I})} \left(-\bP_p^{(-)} + \frac{\sqrt{p} - 1}{2\sqrt{p}}\bm I_{\FF_p} + \frac{\sqrt{p} - 1}{2}\what{\one}\what{\one}^{\top}\right) \bE_{V(G_{p, I})}\right)\\
  &\geq -\lambda_{\max}(\bE_{V(G_{p, I})} \bP_p^{(-)} \bE_{V(G_{p, I})}) + \frac{1}{2} - o(1)
    \intertext{which, if Theorem~\ref{thm:edge-km} applies to the pairs $(\bP_p^{(-)}, \bE_{V(G_{p, I})})$, in which case we must have $\beta = \frac{1}{2^a}$, would be}
  &\geq -\sqrt{\beta(1 - \beta)} - o(1)\\
  &= -\frac{\sqrt{2^a - 1}}{2^a} - o(1), \numberthis
\end{align*}
which would verify Conjecture~\ref{conj:min-eval}.

We note that all of our manipulations from Section~\ref{sec:pf:thm:level1} apply just as well to $\bP^{(-)}_p$ as to $\bP^{(+)}_p$, so this should not be a substantial change to our strategy.
However, it seems technically quite challenging to apply Theorem~\ref{thm:edge-km} in this setting.

We would need stronger bounds on the same traces treated before; following the argument in Section~\ref{sec:pf:thm:level1}, we have that, in general,
\begin{align*}
  T_k
  &\colonequals
  \left| \Tr\left(\left(\bP^{(-)}_p - \frac{1}{2} \bm I_{\FF_p}\right) \left(\bE_{V(G_{p, I})} - \frac{1}{2^a}\bm I_{\FF_p}\right)\right)^{k} \right| \\
  &\leq a + \frac{1}{2^{(a + 1)k}} \left|\sum_{\substack{Z_1, \dots, Z_k \subseteq I \\ Z_i \neq \emptyset}} \Sigma(Z_1, \dots, Z_k)\right| \label{eq:sum-of-sums} \numberthis
  \intertext{and, in Section~\ref{sec:pf:thm:level1}, we proceeded by using the triangle inequality and bounding}
  &\leq a + \frac{1}{2^kp^{k/2}} \max_{\substack{Z_1, \dots, Z_k \subseteq I \\ Z_i \neq \emptyset}} |\Sigma(Z_1, \dots, Z_k)|.\numberthis
\end{align*}
But note that, here, the cancellations posited by Conjecture~\ref{conj:char-sum} cannot help us anymore: they would show that the remaining maximum is $O(\poly(p))$, but, even so, in Theorem~\ref{thm:edge-km} this expression appears with an extra exponential factor of $(2 / \sqrt{\beta(1 - \beta)})^k$.
We would only find that
\begin{equation}
    \left(\frac{2}{\sqrt{\beta(1 - \beta)}}\right)^k T_k \leq \left(\frac{2}{\sqrt{\beta(1 - \beta)}}\right)^ka + \left(\frac{1}{\sqrt{\beta(1 - \beta)}}\right)^k\poly(p).
\end{equation}
Both of the terms here grow as $\exp(\Omega(k))$, failing to satisfy the condition of Theorem~\ref{thm:edge-km}.
(We must take $k \gg \log p$, so $\exp(\Omega(k)) \gg \poly(p)$.)

Thus, we would need a tighter bound on $T_k$.
First, we cannot afford the constant term $a$, which would require making the initial steps of bounding ``nuisance'' terms in $T_k$ more precise.
This aside, the more daunting challenge is in bounding the summation in \eqref{eq:sum-of-sums} before applying the triangle inequality.
Essentially, what this is asking is, beyond showing cancellations \emph{within} each necklace character sum, that we show cancellations \emph{among} these sums when many of them are added together.
Indeed, observing that $1 / \sqrt{\beta(1 - \beta)} = \Theta(2^{a/2})$, one may check that the bound we would need for this second term is
\begin{equation}
    \left|\sum_{\substack{Z_1, \dots, Z_k \subseteq I \\ Z_i \neq \emptyset}} \Sigma(Z_1, \dots, Z_k)\right| \leq 2^{ak/2} \exp(o(k)).
\end{equation}
There are $(2^a - 1)^k \approx 2^{ak}$ terms each of order $\poly(p) = \exp(o(k))$ in the summation, so this is asking for a ``soft'' version of square root cancellations in a sum of necklace character sums.

Similar results are found in number theory; indeed, this is in the spirit of Theorems~\ref{thm:katz-monomial-bound} and \ref{thm:twisted-kloosterman} on sums of Gauss sums and Kloosterman sums, respectively.
But, we expect that this would require a far deeper study of the statistics of necklace character sums than we undertake here.

\section{Numerical experiments}
\label{sec:numerical}

We present some numerical experiments on localizations of the Paley graph to support our main conjectures.

In support of Conjecture~\ref{conj:weak-conv}, we plot the empirical spectral distributions of examples of low-degree localizations of $G_p$ for a large $p$ together with the predicted Kesten-McKay limits in Figure~\ref{fig:example-loc}.

It is reasonable to not find this fully convincing, especially since, once $a \geq 3$, there are many non-isomorphic localizations of degree $a$.
Therefore, we also compute a distributional distance between the empirical spectral distribution and the predicted Kesten-McKay limit, over the unique (up to isomorphism) localization when $a \in \{1, 2\}$ and over \emph{all} localizations when $a = 3$ (of which there are $\Theta(p)$), and plot these distances in Figure~\ref{fig:dist-loc}.
The distance we use is the Kolmogorov distance, mostly for the sake of computational convenience; see \cite{GS-2002-ChoosingBoundingProbabilityMetrics} for a thorough discussion of such choices.

\begin{definition}
    \label{def:kolmogorov}
    The \emph{Kolmogorov distance} between two probability measures $\mu, \nu$ on $\RR$ is
    \begin{equation}
        K(\mu, \nu) \colonequals \sup_{x \in \RR}|\mu([-\infty, x]) - \nu([-\infty, x])|.
    \end{equation}
\end{definition}

\noindent
Conveniently, it is possible to write down a closed-form cumulative distribution function for the Kesten-McKay distribution, which we register here.

\begin{proposition}[c.d.f.\ of Kesten-McKay]
    Let $\mu \colonequals \KM(v)$ with $v \geq 2$.
    Then, for $x \in [-2\sqrt{v - 1}, 2\sqrt{v - 1}]$,
    \begin{align*}
      \mu([-\infty, x])
      &= \mu([-2\sqrt{v - 1}, x]) \\
      &= \frac{1}{2} + \frac{v}{2\pi}\sin^{-1}\left(\frac{x}{2\sqrt{v - 1}}\right) \\
      &\hspace{1cm} - \frac{v - 2}{4\pi}\tan^{-1}\left(\frac{vx + 4(v - 1)}{(v - 2)\sqrt{4(v - 1) - x^2}}\right) \\
      &\hspace{1cm} - \frac{v - 2}{4\pi}\tan^{-1}\left(\frac{vx - 4(v - 1)}{(v - 2)\sqrt{4(v - 1) - x^2}}\right), \numberthis
    \end{align*}
    where when $v = 2$ then the last two terms are to be interpreted as zero.
\end{proposition}
\noindent
With this in hand, it is straightforward to compute the Kolmogorov distance between a Kesten-McKay law and a suitably rescaled e.s.d.\ of a given matrix.

Lastly, in support of Conjecture~\ref{conj:min-eval}, we plot the minimum eigenvalues of localizations of degree $a \in \{1, 2, 3\}$ in Figure~\ref{fig:min-eval} together with the prediction of the Conjecture, the edge of the support of the corresponding Kesten-McKay distribution.
As before, for $a \in \{1, 2\}$ there is only one number to plot, while for $a = 3$ we plot the minimum eigenvalues for all degree 3 localizations (more specifically, plotting the mean and the range over all localizations).

\section*{Acknowledgments}
\addcontentsline{toc}{section}{Acknowledgments}

Thanks to Afonso Bandeira, Emmanuel Kowalski, Dustin Mixon, Daniel Spielman, and Xifan Yu for many pieces of advice and stimulating discussions.
Thanks also to David Gamarnik and participants of the MIT Stochastics and Statistics seminar for several helpful suggestions about quasi-random graphs and connections with Ramsey theory.

\begin{figure}
    \begin{minipage}{0.47\textwidth}
        \includegraphics[scale=0.5]{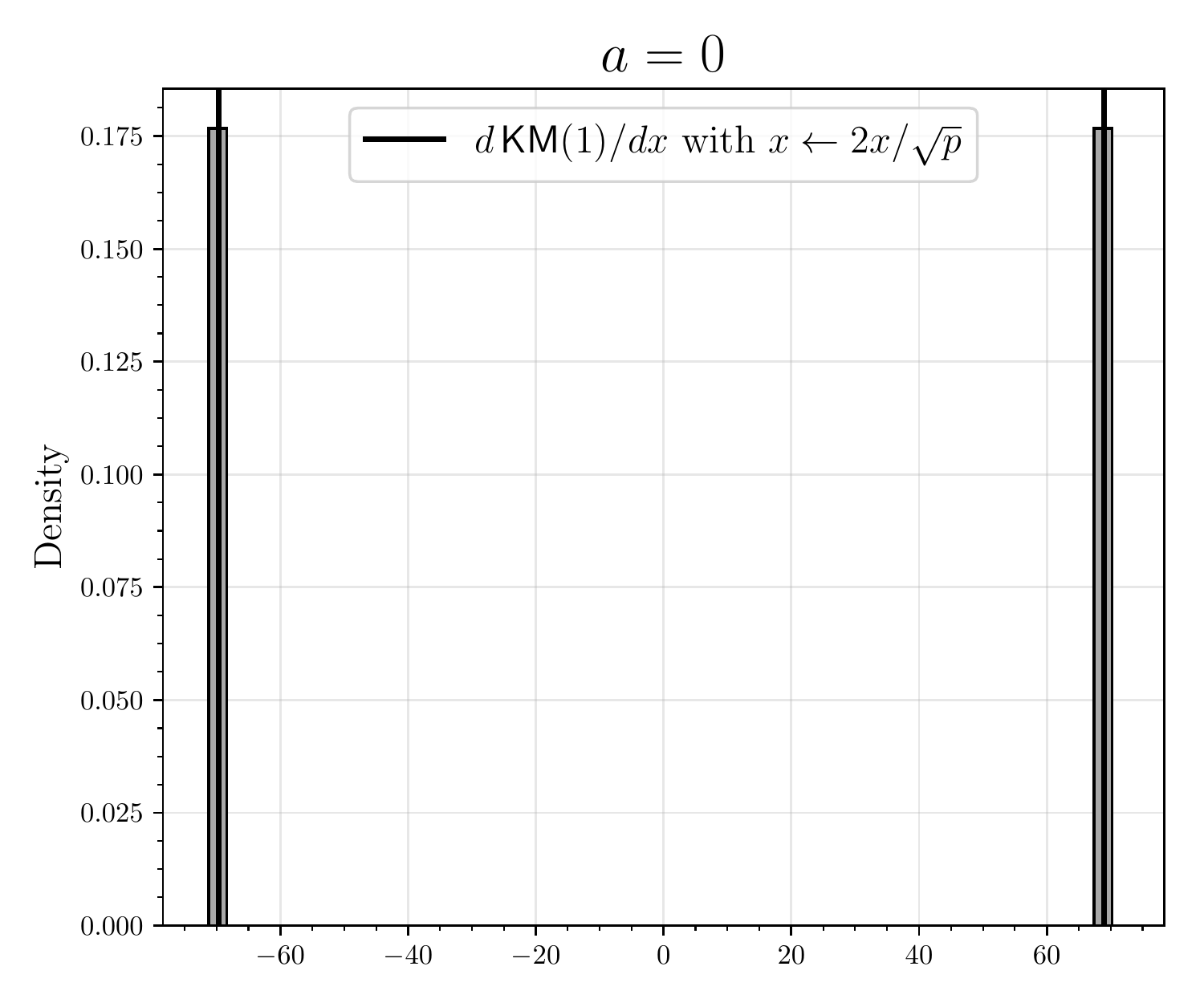}

        \includegraphics[scale=0.5]{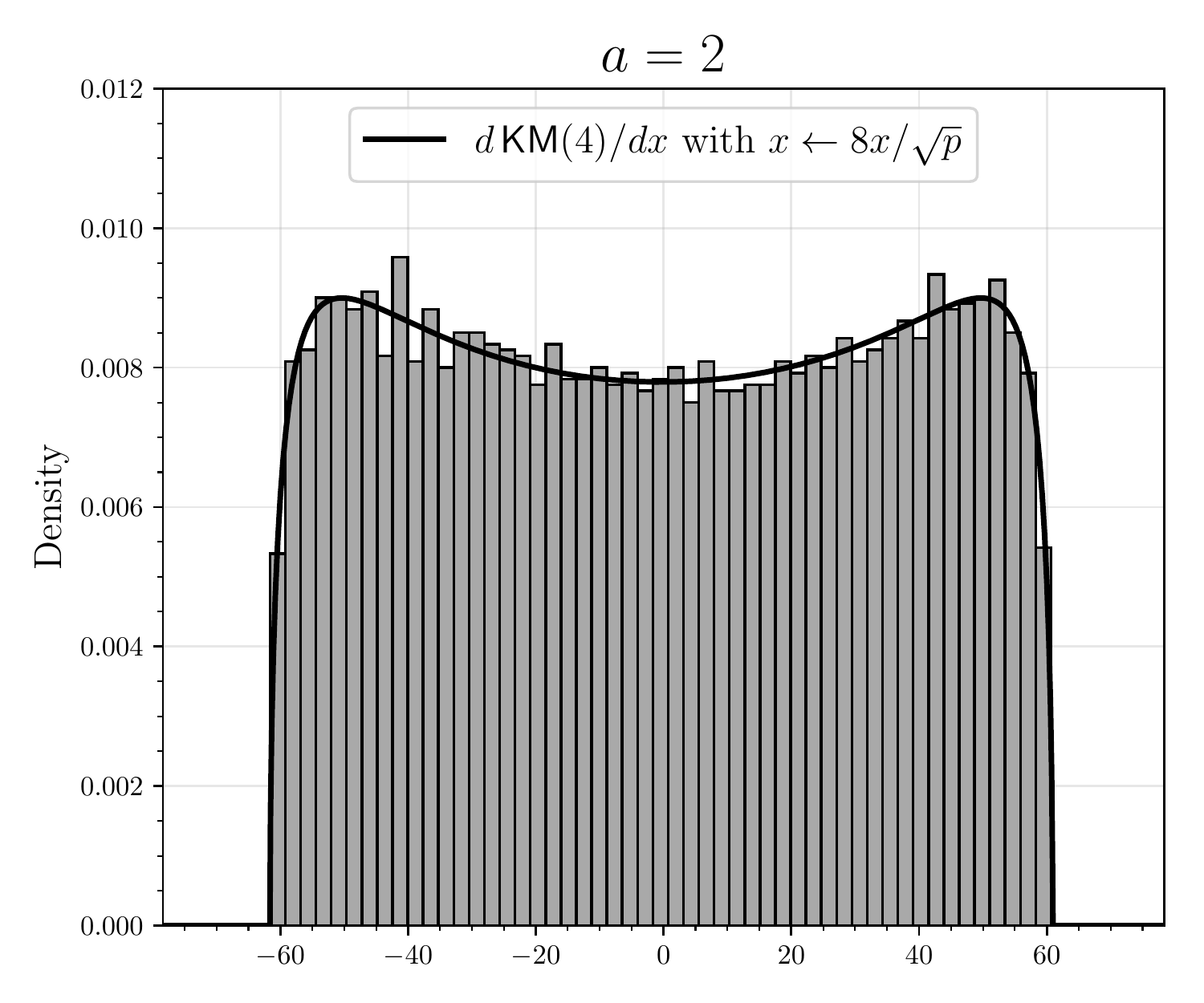}

        \includegraphics[scale=0.5]{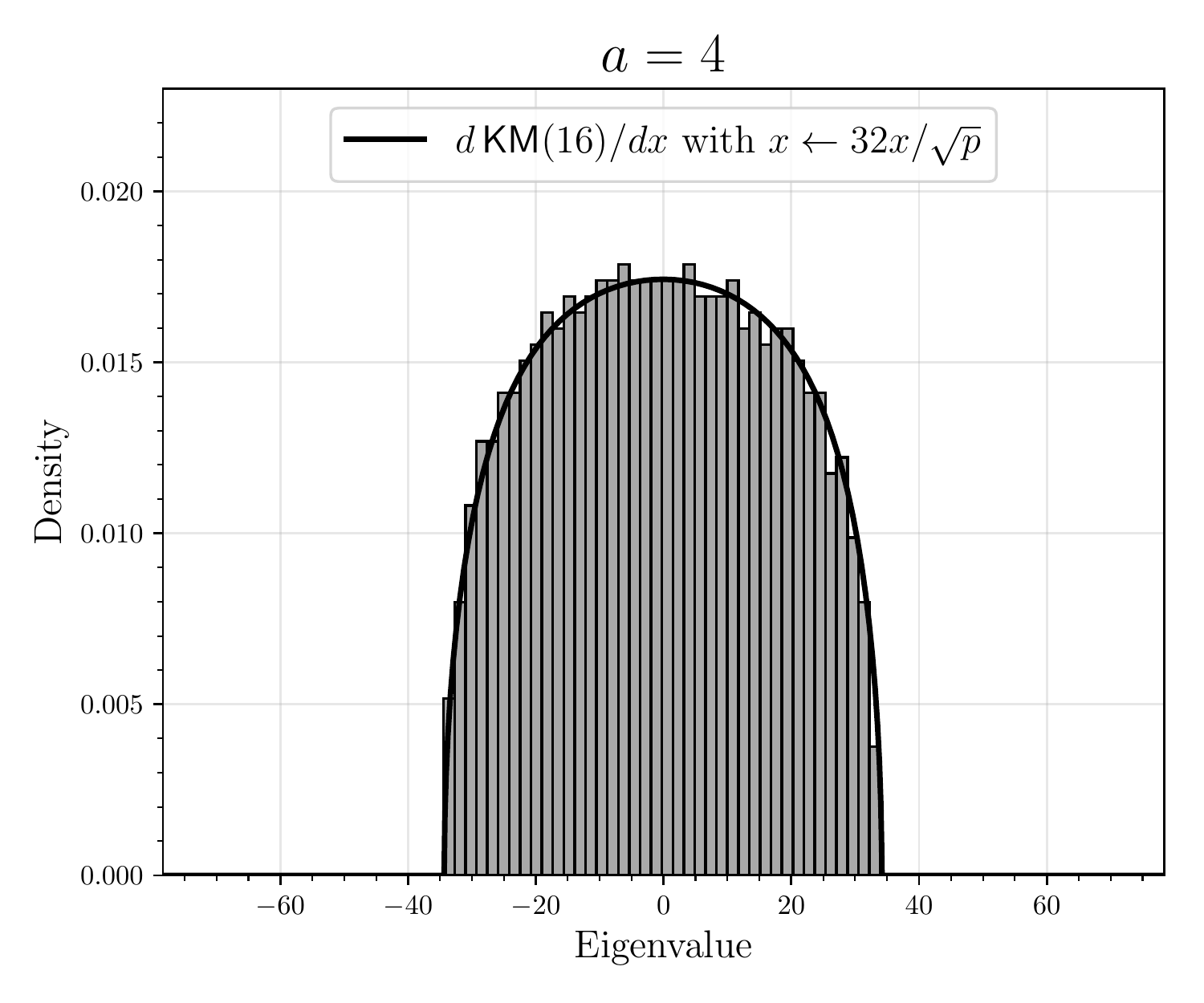}
    \end{minipage}
    \begin{minipage}{0.47\textwidth}
        \includegraphics[scale=0.5]{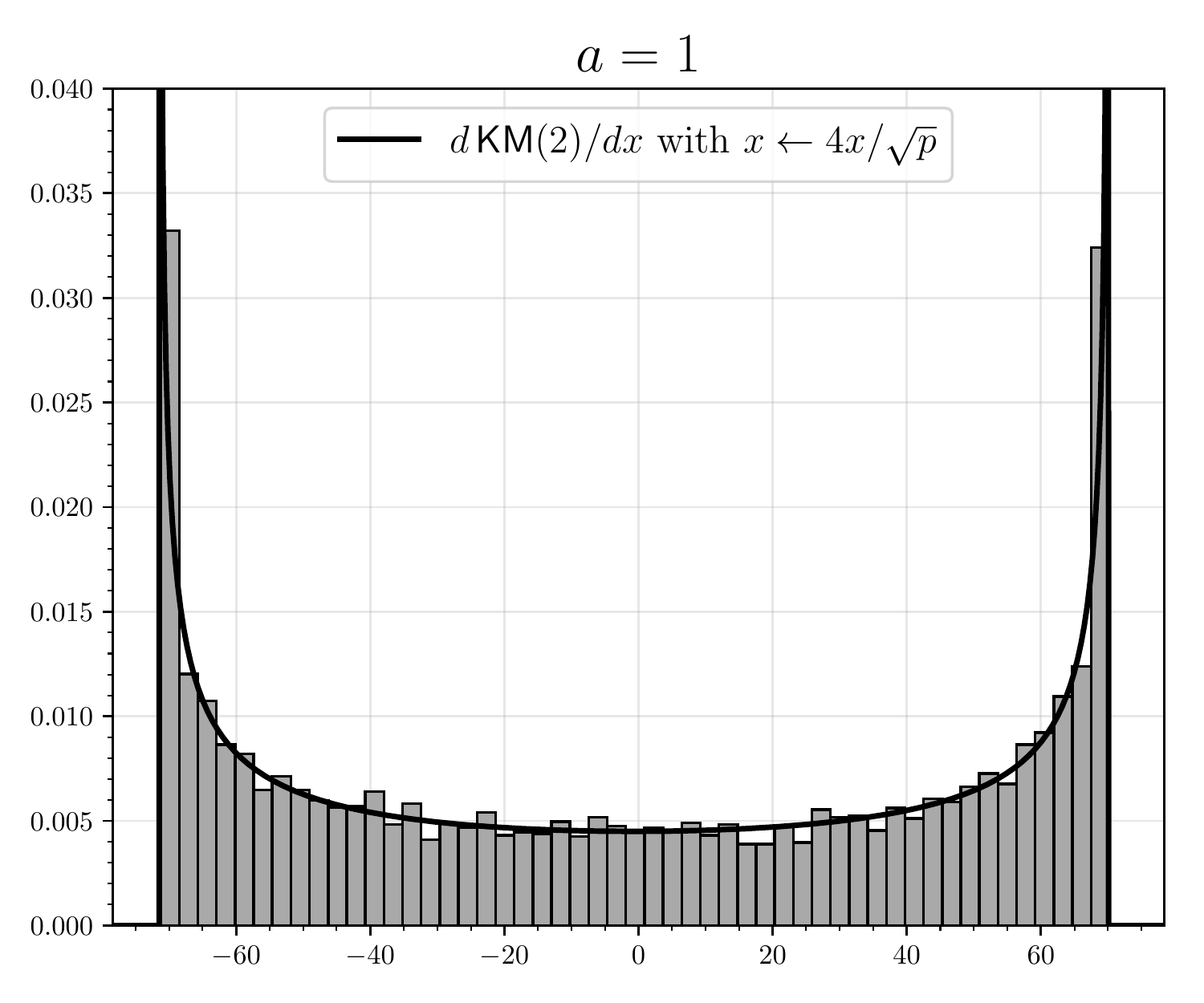}

        \includegraphics[scale=0.5]{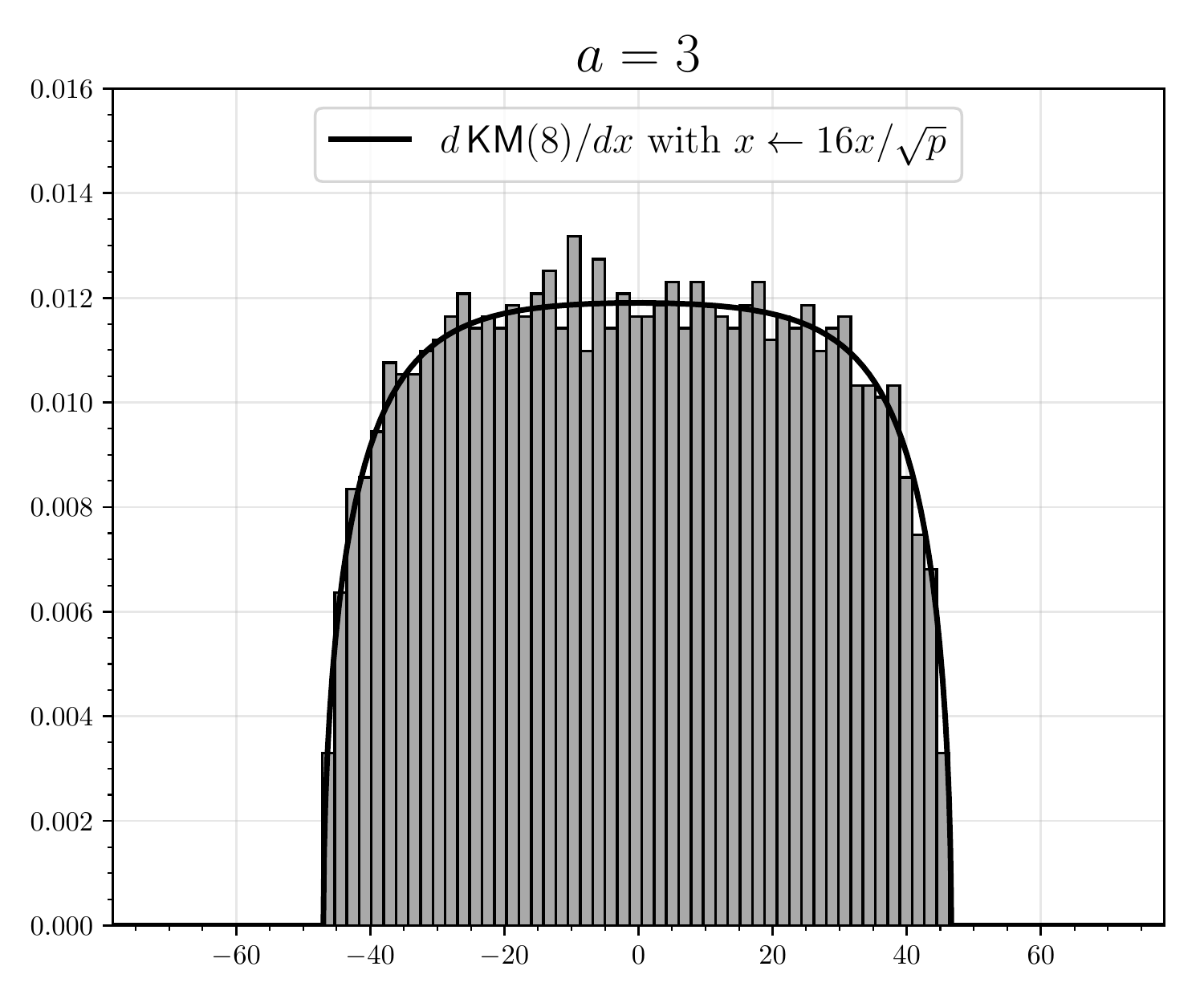}

        \includegraphics[scale=0.5]{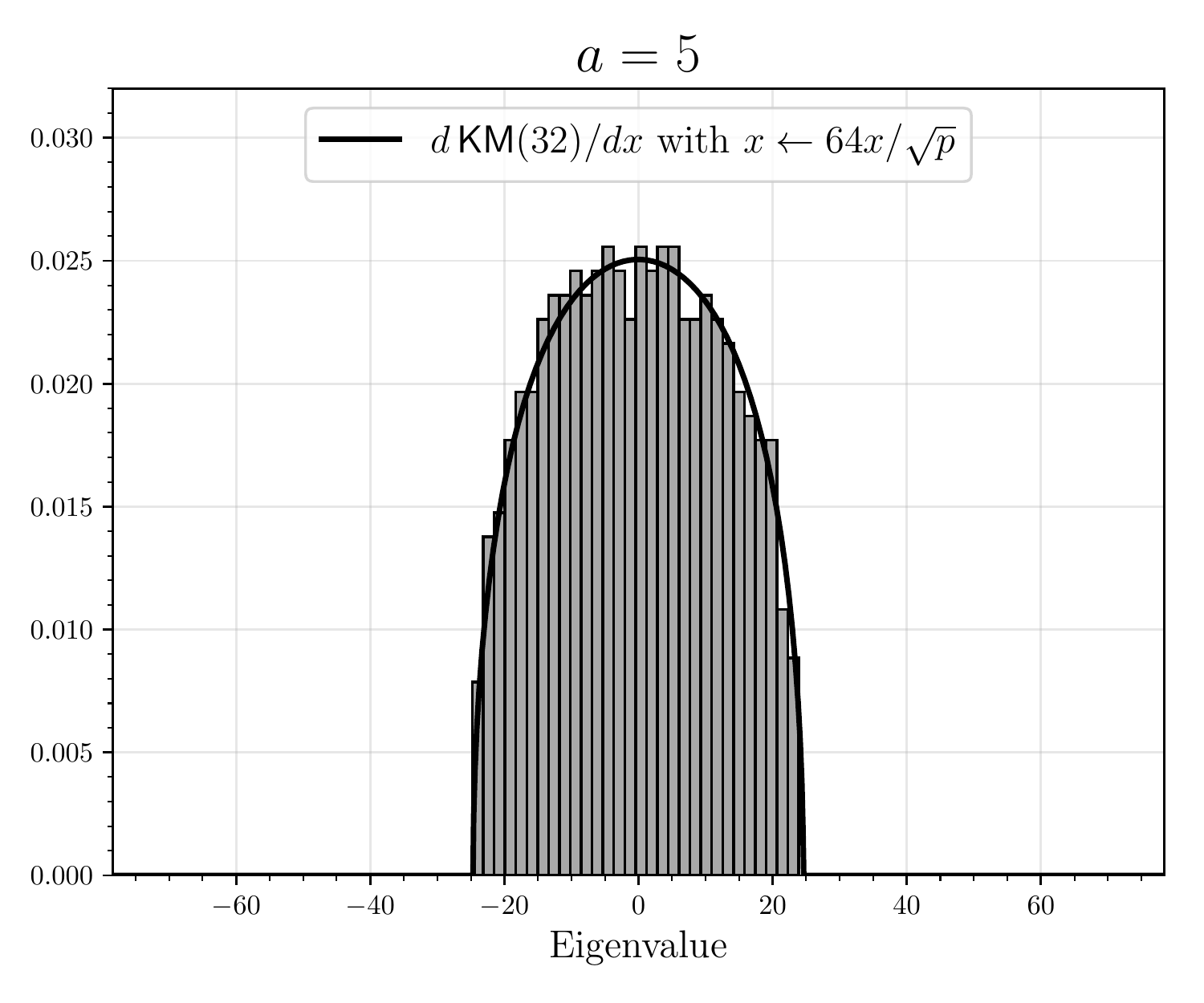}
    \end{minipage}

    \vspace{0.5em}

    \caption{\textbf{Spectra of localizations.} We plot the spectrum of $\bA_{G_{p, I}}$ (with the Perron-Frobenius largest eigenvalue omitted) along with the density of $\KM(2^a)$ suitably rescaled, for $p = \num[group-separator={,}]{20021}$ and one choice of $I$ for each size $a = |I| \in \{0, 1, 2, 3, 4, 5\}$.}
    \label{fig:example-loc}
\end{figure}

\clearpage

\begin{figure}
    \begin{center}
        \includegraphics[scale=0.6]{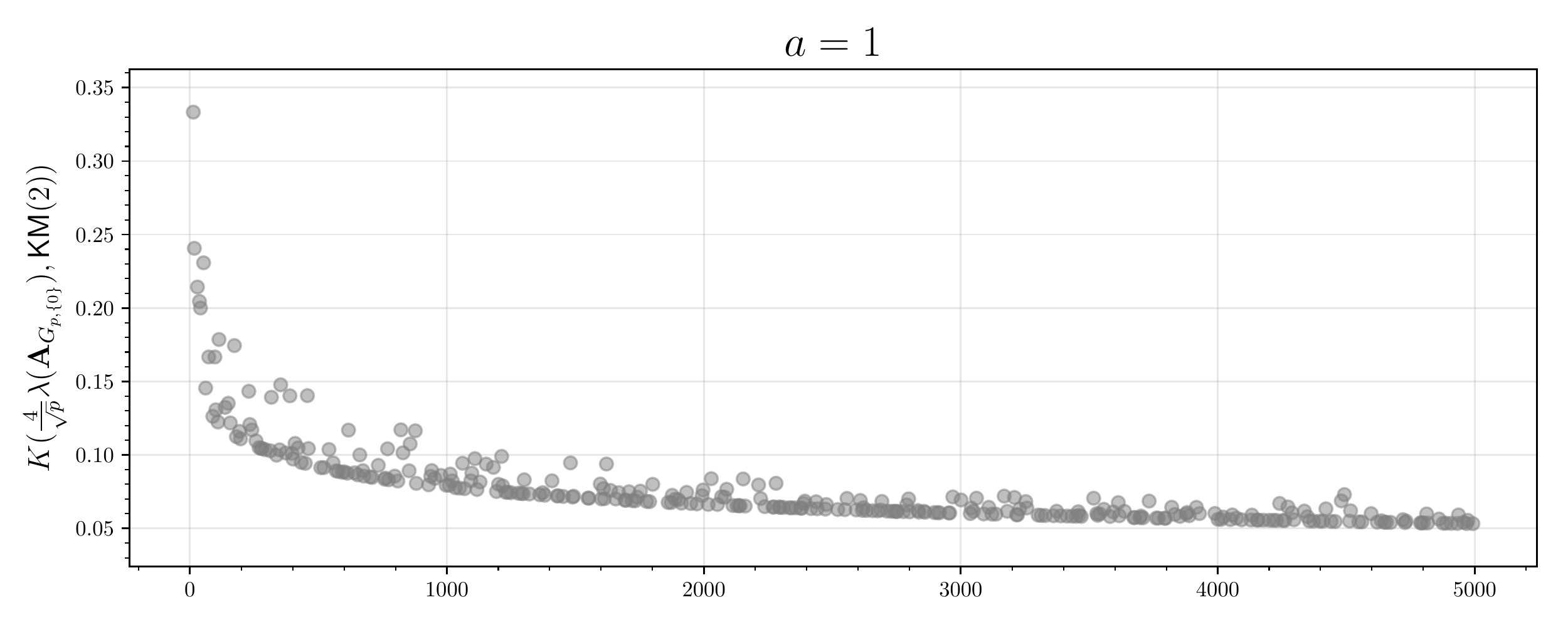}

        \includegraphics[scale=0.6]{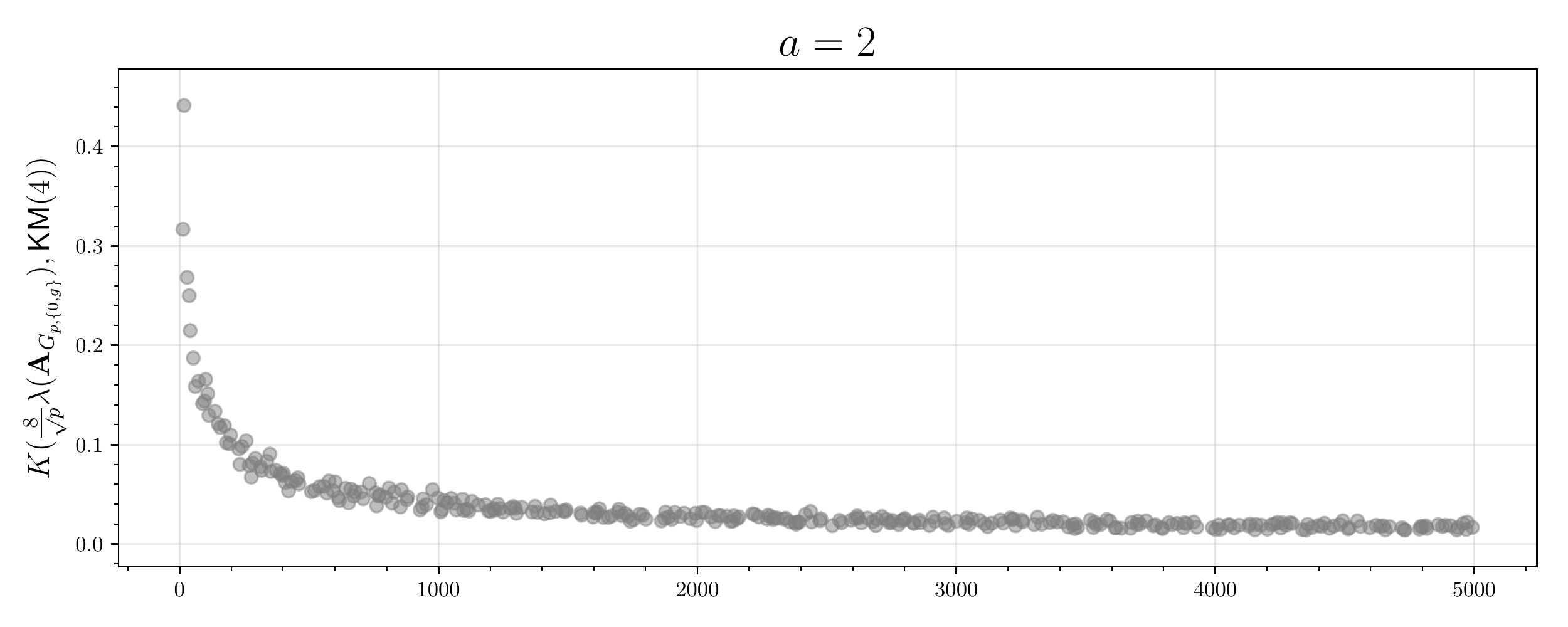}

        \includegraphics[scale=0.6]{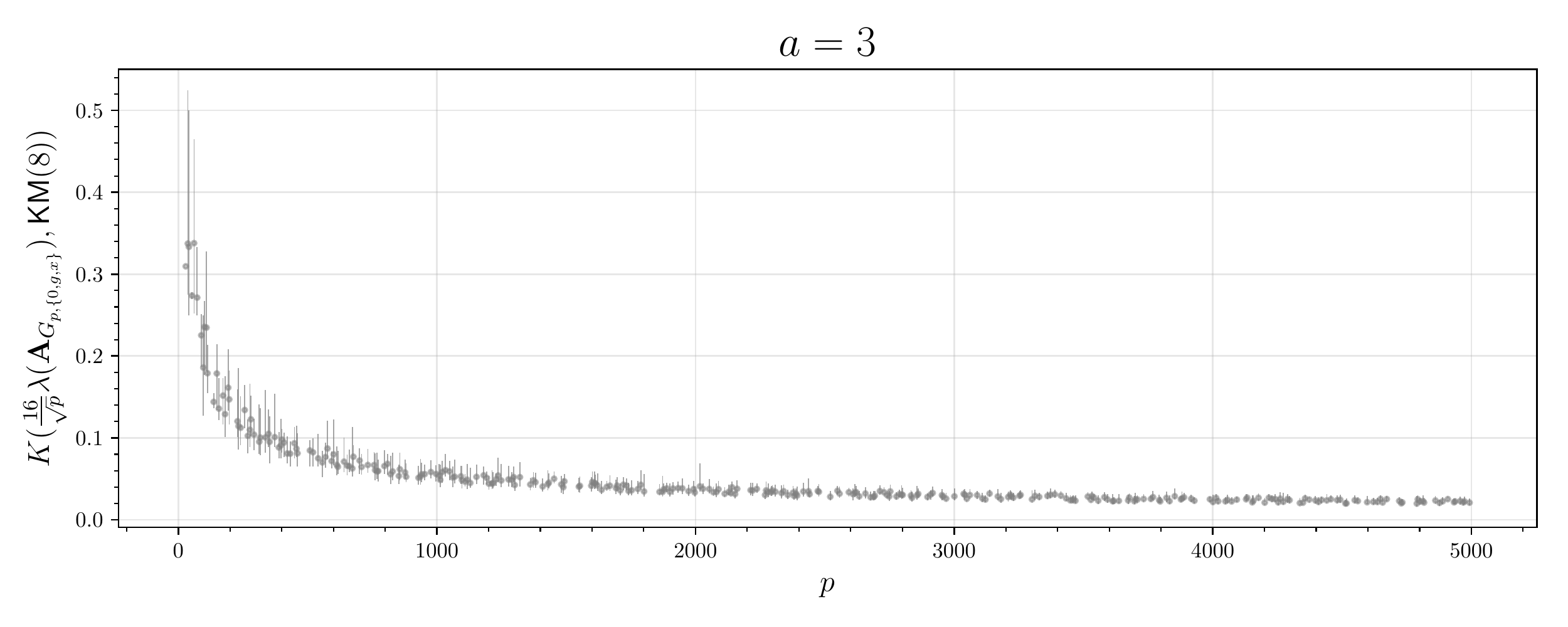}
    \end{center}

    \vspace{-1em}

    \caption{\textbf{Distributional distance of e.s.d.\ of localizations to Kesten-McKay.} We plot the Kolmogorov distance between the empirical spectral distribution of $\frac{2^{a + 1}}{\sqrt{p}} \bA_{G_{p, I}}$ and $\KM(2^a)$ for $a \in \{1, 2, 3\}$ and all $p \leq \num[group-separator={,}]{5000}$. For $a = 1$ (respectively $a = 2$) all $G_{p, I}$ are isomorphic to $G_{p, \{0\}}$ (respectively $G_{p, \{0, g\}}$ for $g$ a generator of $\FF_p^{\times}$), so we plot one value. For $a = 3$, all $G_{p, I}$ are isomorphic to $G_{p, \{0, g, x\}}$ for some $x$, so we plot the mean, minimum, and maximum of the Kolmogorov distance over all $x$ with $\{0, g, x\} \in \sI_3(G_p)$.}
    \label{fig:dist-loc}
\end{figure}

\clearpage

\begin{figure}
    \begin{center}
        \includegraphics[scale=0.6]{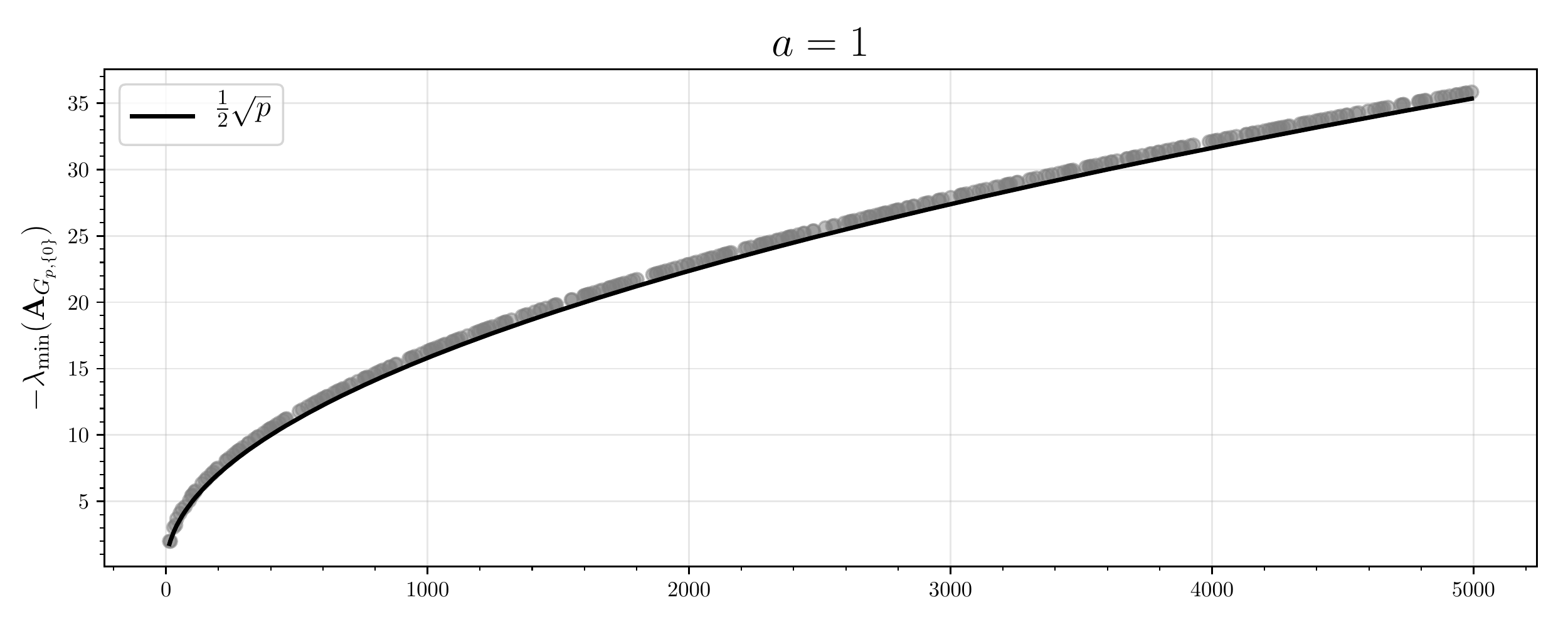}

        \includegraphics[scale=0.6]{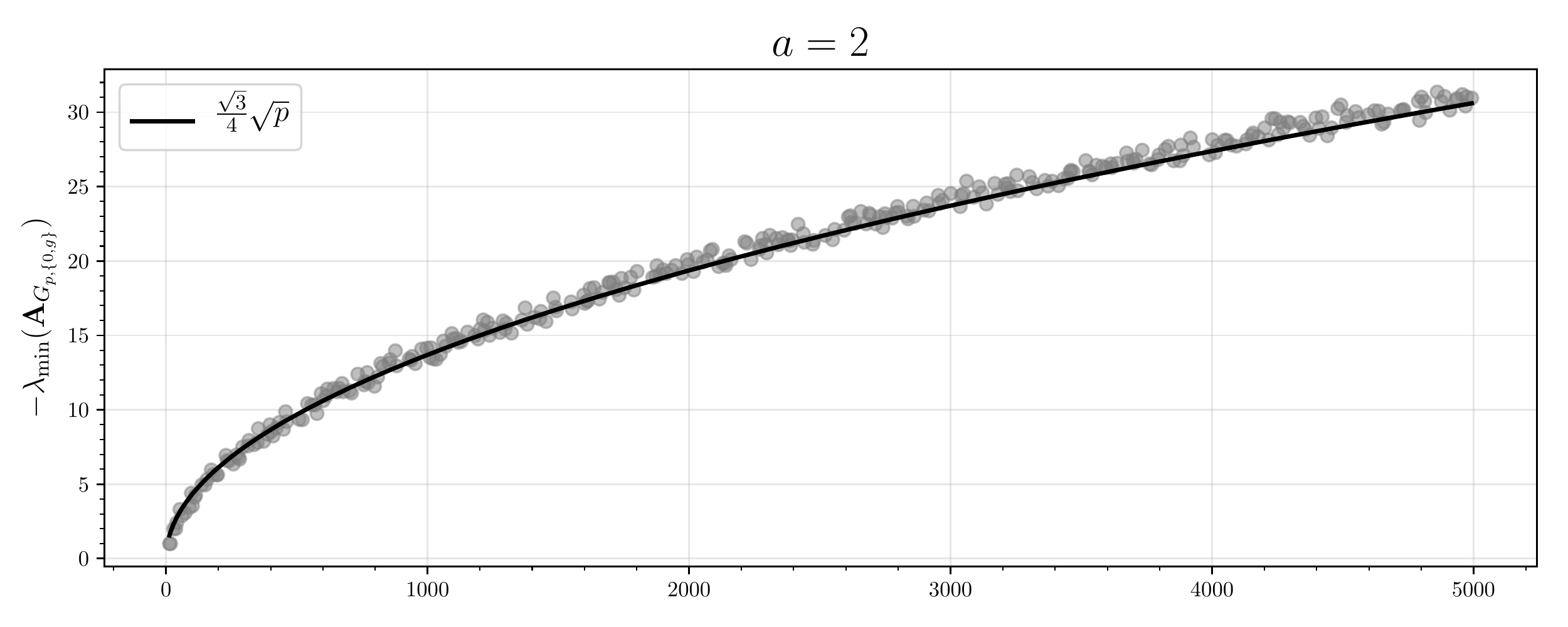}

        \includegraphics[scale=0.6]{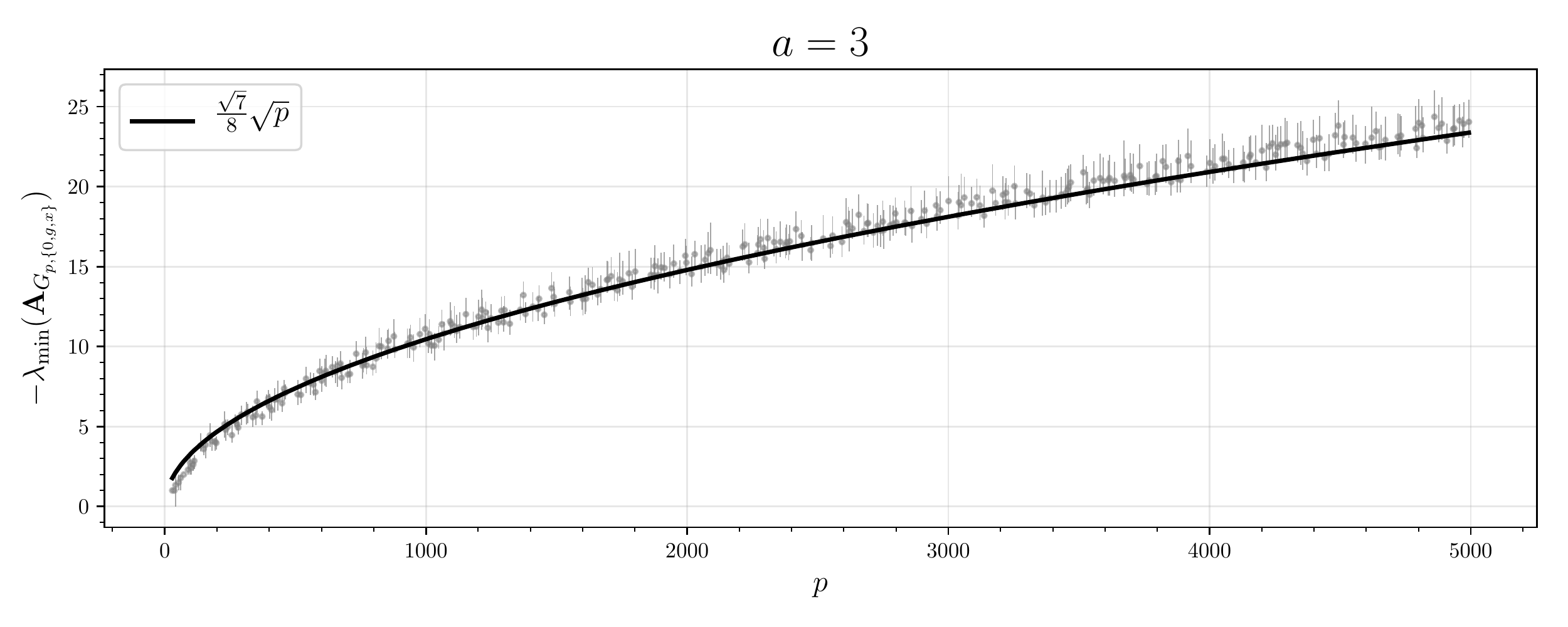}
    \end{center}

    \vspace{-1em}

    \caption{\textbf{Minimum eigenvalues of localizations.} We plot $-\lambda_{\min}(\bA_{G_{p, I}})$ over all $I \in \sI_a(G_p)$ for $p \leq \num[group-separator={,}]{5000}$ and $a \in \{1, 2, 3\}$, along with the corresponding prediction from Conjecture~\ref{conj:min-eval}. The remarks from Figure~\ref{fig:dist-loc} concerning the distinction between $a \in \{1, 2\}$ and $a = 3$ apply here as well.}
    \label{fig:min-eval}
\end{figure}

\clearpage

\addcontentsline{toc}{section}{References}
\bibliographystyle{alpha}
\bibliography{main}

\clearpage

\appendix

\section{\Lovasz\ $\vartheta$ function on degree 2 localization}
\label{app:theta-01}

\begin{figure}
    \begin{center}
        \includegraphics[scale=0.65]{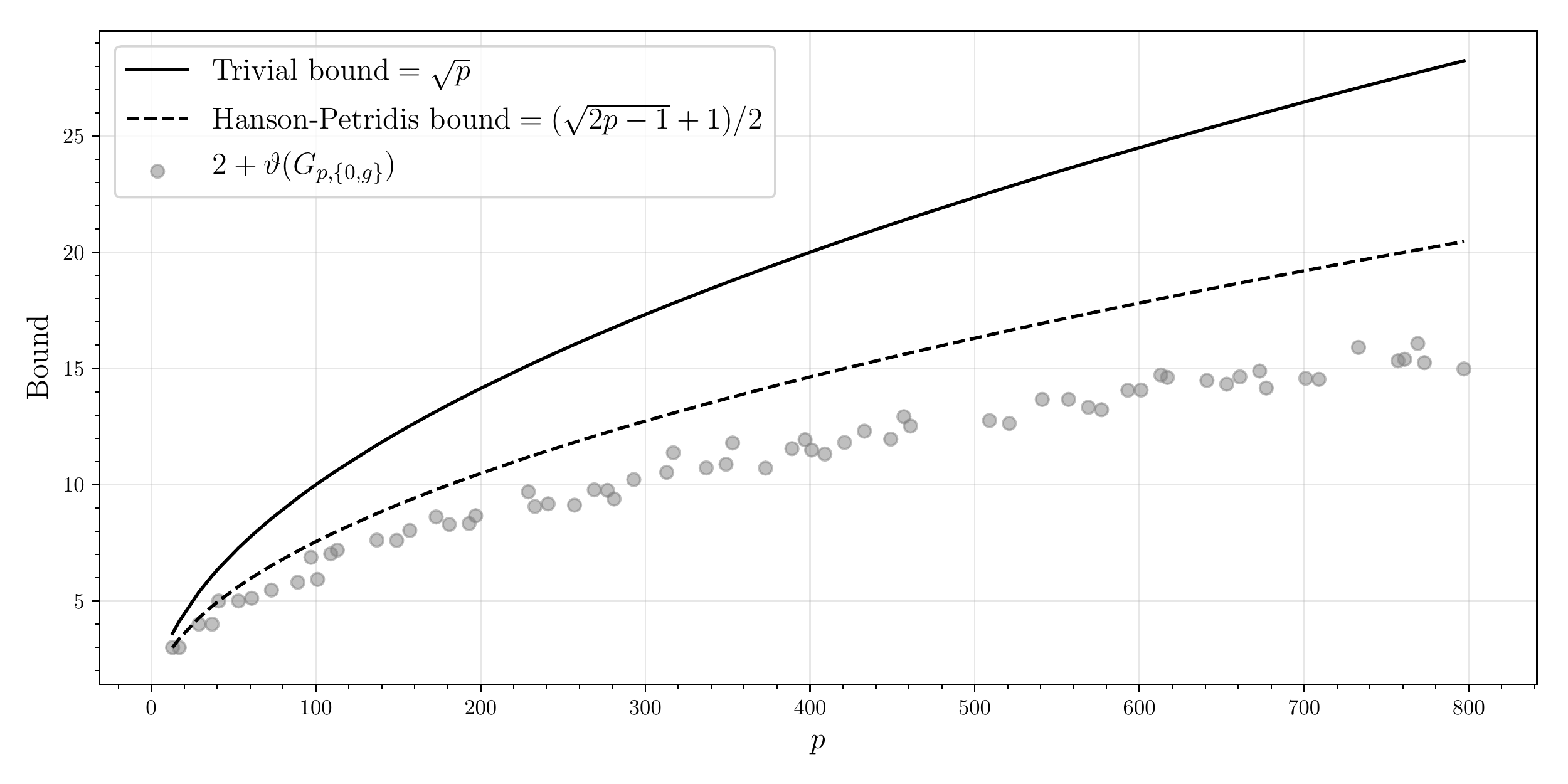}
    \end{center}
    \vspace{-1em}
    \caption{\textbf{Bound from \Lovasz\ $\vartheta$ function on 2-localization.} We plot the bound on $\alpha(G_p)$ from \eqref{eq:lovasz2-bound} together with the ``trivial'' $\sqrt{p}$ bound and the Hanson-Petridis bound of Theorem~\ref{thm:hp-bound} for all $p \leq 800$.}
    \label{fig:lovasz-2}
\end{figure}

We explore empirically a simple \emph{ad hoc} improvement of the proof strategy pursued by \cite{MMP-2019-LPCliquesPaley}.
Recall that their approach applied the \Lovasz\ $\vartheta$ function (equivalently, the degree~2 SOS relaxation) together with degree 1 localization, using the bound
\begin{equation}
    \alpha(G_p) \leq 1 + \max_{I \in \sI_1(G_p)} \vartheta(G_{p, I}) = 1 + \vartheta(G_{p, \{0\}}),
\end{equation}
where we have used that, by Corollary~\ref{cor:transitivity}, all of the $G_{p, \{x\}}$ are isomorphic.
They found empirically that the right-hand side is very close to the Hanson-Petridis bound of Theorem~\ref{thm:hp-bound}.
These empirical experiments were eased by the observation that $G_{p, \{0\}}$ has a circulant adjacency matrix (see the discussion in Section~\ref{sec:equidistribution-level1}), whereby $\vartheta(G_{p, \{0\}})$ may be reduced to a linear rather than semidefinite program.

Consider the same approach with degree 2 localization.
Let $g$ be a multiplicative generator of $\FF_p^{\times}$.
We have
\begin{equation}
    \label{eq:lovasz2-bound}
    \alpha(G_p) \leq 2 + \max_{I \in \sI_2(G_p)} \vartheta(G_{p, I}) = 2 + \vartheta(G_{p, \{0, g\}}),
\end{equation}
where again we have that, by Corollary~\ref{cor:transitivity}, all of the $G_{p, I}$ for $I \in \sI_2(G_p)$ are isomorphic.
Unfortunately, the adjacency matrix of $G_{p, \{0, g\}}$ is \emph{not} circulant, so there is not enough symmetry remaining in this problem to reduce $\vartheta(G_{p, \{0, g\}})$ to a linear program.
Still, unlike higher-degree localization bounds, this bound only calls for solving one semidefinite program for a given value of $p$, and so is relatively computationally tractable.
We compare this bound to the trivial and Hanson-Petridis bounds in Figure~\ref{fig:lovasz-2}, and find that it appears to improve further on the constant in front of $\sqrt{p}$.
We state this formally below.

\begin{conjecture}[Improving on Hanson-Petridis]
    \label{conj:lovasz2}
    For some $\varepsilon > 0$ and for all sufficiently large $p$, $\vartheta(G_{p, \{0, g\}}) \leq (\frac{1}{\sqrt{2}} - \varepsilon)\sqrt{p}$.
\end{conjecture}

The natural approach to Conjecture~\ref{conj:lovasz2} is to build a feasible point for SDP dual to $\vartheta(G_{p, \{0, g\}})$.
We have not succeeded in finding a construction that is tractable to analyze theoretically.
The main difficulty seems to be precisely the lower amount of symmetry (or higher amount of pseudorandomness) in $G_{p, \{0, g\}}$ compared to $G_p$ and $G_{p, \{0\}}$; numerically, the dual variables at optimality do not seem to exhibit any particularly simple structure that we can easily exploit.

\section{Towards higher-degree estimates}
\label{app:higher-degree}

\subsection{One special degree 2 estimate}
\label{app:level2}

We give one example of a degree~2 necklace character sum that actually reduces to the degree~1 necklace character sums treated in Theorem~\ref{thm:level1}.
Unfortunately, this seems to be a special phenomenon thanks to the symmetry of this sum, which we do not expect to apply more generally.
\begin{theorem}
    \label{thm:level2}
    For any distinct $z, z^{\prime} \in \FF_p$,
    \begin{equation}
        |\Sigma(\{z, z^{\prime}\}, \dots, \{z, z^{\prime}\})| \leq kp^{(k + 1) / 2} + 2^k p^{k / 2} = O_k(p^{(k + 1) / 2}).
    \end{equation}
\end{theorem}
\begin{proof}
    Let us write $\Sigma \colonequals \Sigma(\{z, z^{\prime}\}, \dots, \{z, z^{\prime}\})$.
By changing variables $x_i \leftarrow x_i + z^{\prime}$ in $\Sigma$, we may assume without loss of generality that $z^{\prime} = 0$ (and thus $z \neq 0$).
We then have
\begin{align*}
  \Sigma
  &= \sum_{x_1, \dots, x_k \in \FF_p} \chi\big(x_1(x_1 - z)(x_1 - x_2) \cdots x_k(x_k - z)(x_k - x_1)\big) \\
  &= \sum_{x_1, \dots, x_k \in \FF_p^{\times}} \chi\big(x_1(x_1 - z)(x_1 - x_2) \cdots x_k(x_k - z)(x_k - x_1)\big) \\
  &= \sum_{x_1, \dots, x_k \in \FF_p^{\times}} \chi\big((1 - zx_1^{-1})(x_1^{-1} - x_2^{-1}) \cdots (1 - zx_k^{-1})(x_k^{-1} - x_1^{-1})\big) \\
  &= \chi(z)^k\sum_{x_1, \dots, x_k \in \FF_p^{\times}} \chi\big((z^{-1} - x_1)(x_1 - x_2) \cdots (z^{-1} - x_k)(x_k - x_1)\big) \numberthis
\end{align*}

This sum only differs by the restriction to $\FF_p^{\times}$ from the degree 1 necklace character sum
\begin{align*}
  \Sigma(\{z^{-1}\}, \dots, \{z^{-1}\})
  &= \sum_{x_1, \dots, x_k \in \FF_p}\chi\big((z^{-1} - x_1)(x_1 - x_2) \cdots (z^{-1} - x_k)(x_k - x_1)\big)
    \intertext{and our plan will be to bound the difference between these two. To that end, we view the above in linear-algebraic terms as}
  &= \Tr(\bD\bS)^k, \numberthis
\end{align*}
where $\bD \in \RR^{\FF_p \times \FF_p}$ is the diagonal matrix with diagonal entries $D_{xx} = \chi(z^{-1} - x)$ and $\bS$ is the $\{\pm 1\}$-valued adjacency matrix of $G_p$.
Let $\widetilde{\bD}$ be $\bD$ with the row and column indexed by $0 \in \FF_p$ set to zero.
Then, we have
\begin{equation}
    \bD = \chi(z)\be_0\be_0^{\top} + \widetilde{\bD}.
\end{equation}
Note that the sum we were originally interested in is
\begin{equation}
    \Sigma = \chi(z)^k \Tr(\widetilde{\bD}\bS)^k.
\end{equation}

As in our arguments in Section~\ref{sec:prf:thm:conj-equiv}, let us define
\begin{align}
  \bH_0 &\colonequals \chi(z)\be_0\be_0^{\top}, \\
  \bH_1 &\colonequals \widetilde{\bD}.
\end{align}
Then, we have
\begin{equation}
  \Sigma(\{z^{-1}\}, \dots, \{z^{-1}\}) = \chi(z)^k\Sigma + \sum_{\substack{s_1, \dots, s_k \in \{0, 1\} \\ \text{some } s_i = 0}} \Tr(\bH_{s_1} \cdots \bH_{s_k}).
\end{equation}
Rearranging, we have
\begin{align*}
  |\Sigma|
  &\leq |\Sigma(\{z^{-1}\}, \dots, \{z^{-1}\})| + \sum_{\substack{s_1, \dots, s_k \in \{0, 1\} \\ \text{some } s_i = 0}} |\Tr(\bH_{s_1}\bS \cdots \bH_{s_k}\bS)|
  \intertext{Here, the first quantity is bounded by Theorem~\ref{thm:level1}. Each term in the second sum may be written $|\be_0^{\top} \bH \be_0| \leq \|\bH\|$ for some $\bH$ with $\|\bH\| \leq \|\bS\|^{k} \leq p^{k / 2}$. Thus, we find}
  &\leq kp^{(k + 1) / 2} + 2^k p^{k / 2}, \numberthis
\end{align*}
where we have substituted in the explicit bound \eqref{eq:level1-explicit-bound} from our proof of Theorem~\ref{thm:level1}, completing the proof.
\end{proof}

\subsection{General reduction to Gauss sum polynomials}

We mention one approach that might prove fruitful for the general case of Conjecture~\ref{conj:char-sum}, which generalizes our first approach to the proof of Theorem~\ref{thm:level1}.
We will show that we can rewrite a general necklace character sum as a large polynomial in Jacobi sums, and therefore in angles of Gauss sums.

The following standard result expresses that the Jacobi sums are the multiplicative Fourier transform of additive translation of a multiplicative character.
\begin{proposition}
    \label{prop:jacobi-ft}
    For all $x \in \FF_p^{\times}$,
    \begin{equation}
        \phi(1 - x) = \frac{1}{p - 1}\sum_{\psi \in \what{\FF_p^{\times}}} J(\phi, \psi) \overline{\psi}(x).
    \end{equation}
\end{proposition}

Suppose we are interested in controlling a general necklace character sum
\begin{equation}
    \Sigma(Z_1, \dots, Z_k) = \sum_{x_1, \dots, x_k \in \FF_p} \chi(x_2 - x_1) \cdots \chi(x_k - x_{k - 1}) \chi(x_1 - x_k) \prod_{i = 1}^k \prod_{z \in Z_i} \chi (x_i - z).
\end{equation}
For constant degree $a$, since $|Z_1 \cup \cdots \cup Z_k| \leq a$, by a suitable change of variables we may assume without loss of generality that $0 \notin Z_i$ for all $i$.
By the argument from Theorem~\ref{thm:level2}, it also suffices to control the restriction to $x_i \in \FF_p^{\times}$, so we will look at the quantity
\begin{equation}
    \widetilde{\Sigma} \colonequals \sum_{x_1, \dots, x_k \in \FF_p^{\times}} \chi(x_2 - x_1) \cdots \chi(x_k - x_{k - 1}) \chi(x_1 - x_k) \prod_{i = 1}^k \prod_{z \in Z_i} \chi (x_i - z).
\end{equation}

Let us define $z_{ij}$ such that $Z_i = \{z_{i1}, \dots, z_{im_i}\}$ for $m_i = |Z_i|$.
We will reserve the indices $i, j$ for enumerating these variables below.
We also assume all index arithmetic is performed modulo $k$.
We apply Proposition~\ref{prop:jacobi-ft} first to each of the $\chi(x_b - z)$ terms and then to each of the $\chi(x_b - x_{b - 1})$, obtaining
\begin{align*}
  \widetilde{\Sigma}
  &= \prod_{i, j} \chi(z_{ij})\frac{1}{(p - 1)^{\sum_{i = 1}^k m_i}}\sum_{\psi_{ij} \in \what{\FF_p^{\times}}}\sum_{x_1, \dots, x_k \in \FF_p^{\times}} \prod_{b = 1}^k \chi(x_b - x_{b - 1}) \prod_{i, j} J(\chi, \psi_{ij}) \overline{\psi_{ij}}(z_{ij}^{-1}x_i) \\
  &= \prod_{i, j} \chi(z_{ij})\frac{1}{(p - 1)^{k + \sum m_i}}\sum_{\psi_{ij} \in \what{\FF_p^{\times}}}\sum_{\phi_1, \dots, \phi_k \in \what{\FF_p^{\times}}} \sum_{x_1, \dots, x_k \in \FF_p^{\times}} \\
  &\hspace{2cm} \prod_{b = 1}^k \overline{\phi_b}(x_bx_{b - 1}^{-1})\chi(x_b) J(\chi, \phi_b) \prod_{i, j} J(\chi, \psi_{ij}) \overline{\psi_{ij}}(z_{ij}^{-1}x_i) \\
  &= \prod_{i, j} \chi(z_{ij})\frac{1}{(p - 1)^{k + \sum m_i}}\sum_{\psi_{ij} \in \what{\FF_p^{\times}}} \sum_{\phi_1, \dots, \phi_k \in \what{\FF_p^{\times}}} \\
  &\hspace{2cm} \prod_{i, j} \psi_{ij}(z_{ij}) J(\chi, \psi_{ij}) \prod_{b = 1}^k J(\chi, \phi_b) \left(\sum_{x \in \FF_p^{\times}} \left(\overline{\phi_b}\phi_{b - 1}\chi \prod_{j = 1}^{m_b} \overline{\psi_{bj}}\right) (x)\right)
    \intertext{Now, each remaining sum is zero unless $\phi_{b - 1} = \phi_b \chi \prod_{j = 1}^{m_b} \psi_{bj}$, in which case it is $p - 1$. Thus, we have}
  &= \prod_{i, j} \chi(z_{ij})\frac{1}{(p - 1)^{\sum m_i}} \sum_{\phi \in \what{\FF_p^{\times}}} \sum_{\substack{\psi_{ij} \in \what{\FF_p^{\times}} \\ \prod_{i, j} \psi_{ij} = \chi^k}}  \\
  &\hspace{2cm} \prod_{i, j} \psi_{ij}(z_{ij}) J(\chi, \psi_{ij}) \prod_{b = 0}^{k - 1}J\left(\chi, \phi\chi^b \prod_{1 \leq i \leq b} \prod_{j = 1}^{m_i} \psi_{ij}\right). \numberthis
\end{align*}
Finally, we may use Proposition~\ref{prop:jacobi} on each Jacobi sum to obtain a large sum of products of angles of Gauss sums, albeit ``twisted'' by the character evaluations $\psi_{ij}(z_{ij})$.
(We may handle the constraint in our summation by setting, say, $\psi_{11} = \chi^k \prod_{(i, j) \neq (1, 1)} \overline{\psi_{ij}}$, though this would break the appealing symmetry of our expression above.)
That detail aside, this is quite similar to the situation treated by Katz's bound from Theorem~\ref{thm:katz-monomial-bound}, except that we have a ``multi-dimensional'' average of angles of Gauss sums, summing over $\sum_{i = 1}^k |Z_i| - 1$ many multiplicative characters $\psi_{ij}$ instead of just one.
Unfortunately, to the best of our knowledge, Katz's bound has not been generalized to this setting.

\section{Limitations on restricted isometry property bounds}
\label{app:rip}

Up to rescaling, the matrices $\bP_p^{(\pm)}$ are the Gram matrices of \emph{unit norm tight frames} associated to the Paley graph, and very close to the Gram matrices of the \emph{Paley equiangular tight frame~(ETF)}, which are rescaled projection matrices with one more dimension and with all off-diagonal entries having equal magnitude.
The work \cite{BFMW-2013-DeterministicRIP} investigated the \emph{restricted isometry property (RIP)} for these matrices, which amounts to showing that all small principal submatrices are uniformly close to a multiple of the identity.
The reader may consult their exposition and especially their Theorem~23 for further details on this connection.

This is a similar question to the one we have considered, except that it asks to consider the spectrum of \emph{arbitrary} submatrices of $\bP^{(\pm)}_p$ (or, by Proposition~\ref{prop:paley-spectral} equivalently, of $\bA_{G_p}$) rather than the special submatrices corresponding to localizations or, per Remark~\ref{rem:low-deg}, general low-degree subgraphs.
In this appendix, we give an example showing that our quite precise conjectures about the spectra of localizations \emph{do not} in general extend to arbitrary induce subgraphs.
While the size of subgraph or submatrix we consider, of order $\Theta(p)$, is not the same as the smaller $o(p)$ submatrices of interest for RIP, this gives some evidence that our approach may not be the correct one for studying RIP for Paley ETFs.

We will produce a suitable induced subgraph of $G_p$ that behaves very differently from the $G_{p, I}$ of similar size.
Let $F_p$ be the induced subgraph on the ``quartic residues'' or non-zero fourth powers modulo $p$.
If $g$ is a multiplicative generator of $\FF_p^{\times}$, these are $g^4, g^8, \dots, g^{p - 1} = 1$.
Thus $|V(F_p)| = \frac{p - 1}{4}$, so $F_p$ has as many vertices as a degree 2 localization of $G_p$, all of which are isomorphic to $G_{p, \{0, g\}}$.

Conveniently, as in the discussion of $G_{p, \{0\}}$ in Section~\ref{sec:equidistribution-level1}, $\bA_{F_p}$ is circulant when the vertices are ordered as above, since $\chi(g^{4b} - g^{4a}) = \chi(g^{4(b - a)} - 1)$, only depending on the value of $b - a$ modulo $\frac{p - 1}{4}$.
Thus, again the eigenvalues may be computed directly by a Fourier transform,
\begin{equation}
    \sum_{j = 0}^{(p - 1) / 4 - 1} \One\{g^{4j} - 1 \in \SS_p\}\, e_{p - 1}(4aj) \text{ for } a = 1, \dots, \frac{p - 1}{4}.
\end{equation}
Each function $g^j \mapsto e(\frac{aj}{p - 1})$ is a multiplicative character of $\FF_p^{\times}$, so we may equivalently express the eigenvalues as a small correction of the character sums
\begin{equation}
    S(\chi, \psi) \colonequals \sum_{y \in \FF_p} \chi(y^4 - 1) \psi(y^4) \text{ for } \psi \in \what{\FF_p^{\times}},
\end{equation}
where each eigenvalue will be counted four times when these are computed over all $\psi \in \what{\FF_p^{\times}}$.

In Figure~\ref{fig:rip}, we illustrate that the e.s.d.\ of $F_p$ has a very different-looking shape from that of $G_{p, \{0, g\}}$.
We also use the above observations to compute the minimum eigenvalue for large $p$, and show that it is typically \emph{smaller} (that is, more negative, by a leading order $\Theta(\sqrt{p})$ amount) than the prediction of Conjecture~\ref{conj:min-eval}.
We illustrate this in Figure~\ref{fig:rip-min-eval}: $-\lambda_{\min}(\bA_{F_p})$ clearly exceeds the prediction of Conjecture~\ref{conj:min-eval} of $\frac{\sqrt{3}}{4}\sqrt{p}$, and appears instead to scale with only very small fluctuations as $\frac{1}{2}\sqrt{p}$.
The following thus seems very plausible.

\begin{conjecture}
    We have
    \begin{equation}
        \lim_{p \to \infty} \frac{\lambda_{\min}(\bA_{F_p})}{\sqrt{p}} = -\frac{1}{2}.
    \end{equation}
\end{conjecture}

\begin{remark}
    Despite this counterexample, it is not out of the question that the combinatorial techniques developed to prove Theorem~\ref{thm:edge-km} in \cite{Kunisky-2023-MANOVAProductProjections} could still be used to prove RIP for frames related to the Paley graph, especially since the same calculations can be carried out even for pairs of projections that are not necessarily asymptotically free.
\end{remark}

\clearpage
\begin{figure}
    \begin{center}
        \begin{minipage}{0.47\textwidth}
            \includegraphics[scale=0.5]{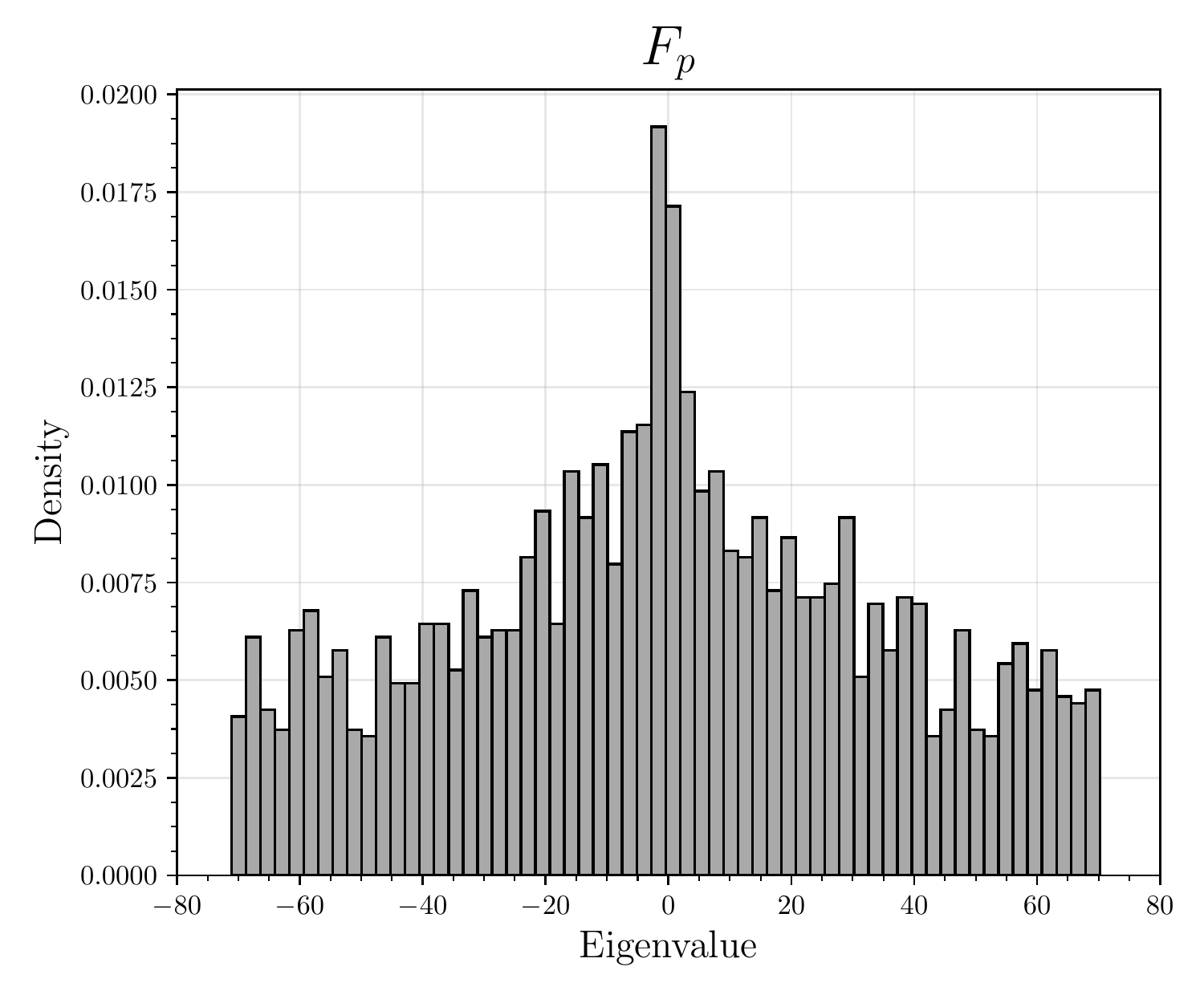}
        \end{minipage}
        \begin{minipage}{0.47\textwidth}
            \includegraphics[scale=0.5]{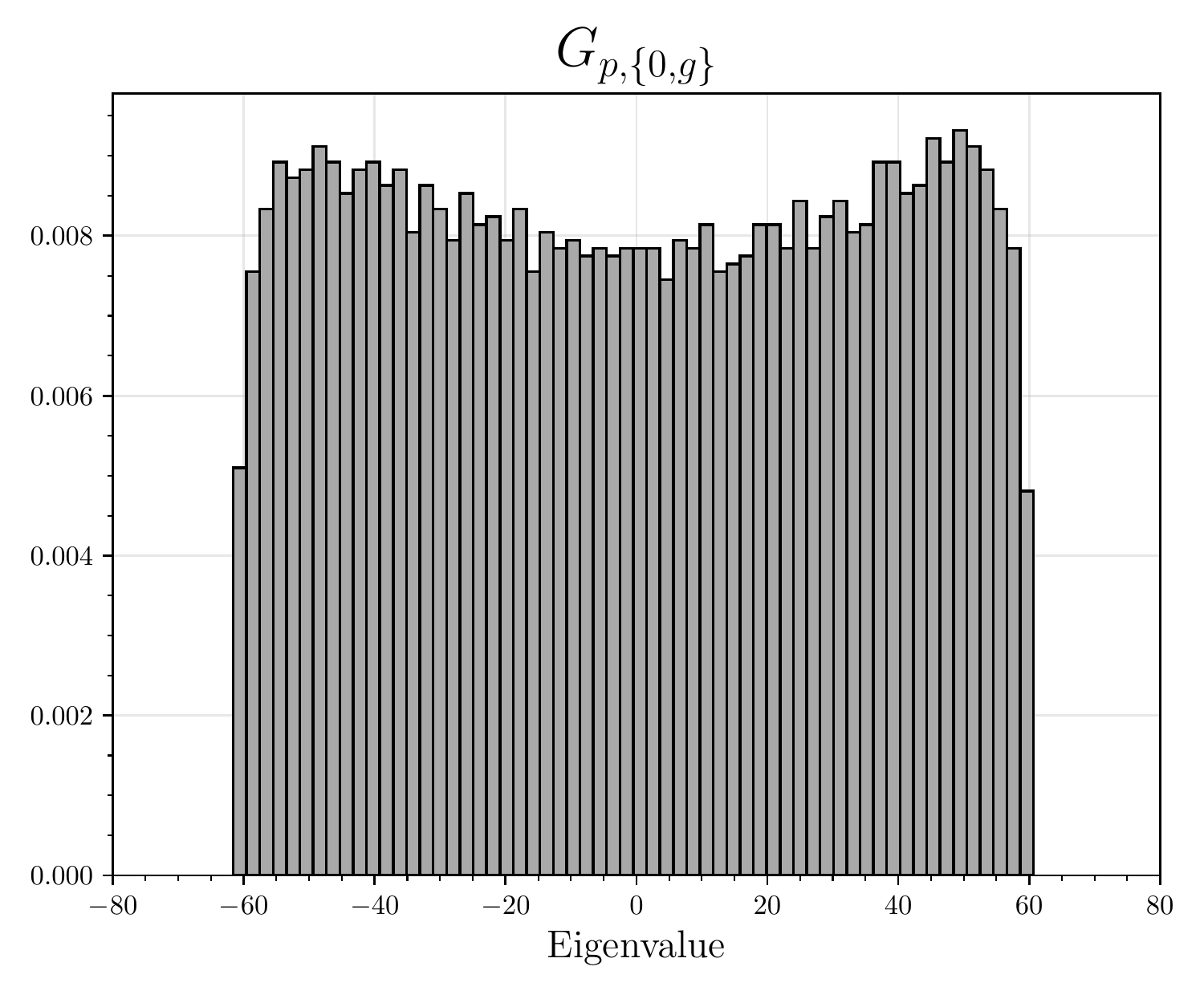}
        \end{minipage}
    \end{center}

    \caption{\textbf{Induced subgraph with non-Kesten-McKay e.s.d.} We compare the empirical spectral distribution of the induced subgraph $F_p$ of $G_p$ discussed in Appendix~\ref{app:rip} with that of $G_{p, \{0, g\}}$ for $p = \num[group-separator={,}]{20021}$.}
    \label{fig:rip}
\end{figure}

\begin{figure}
    \begin{center}
        \includegraphics[scale=0.6]{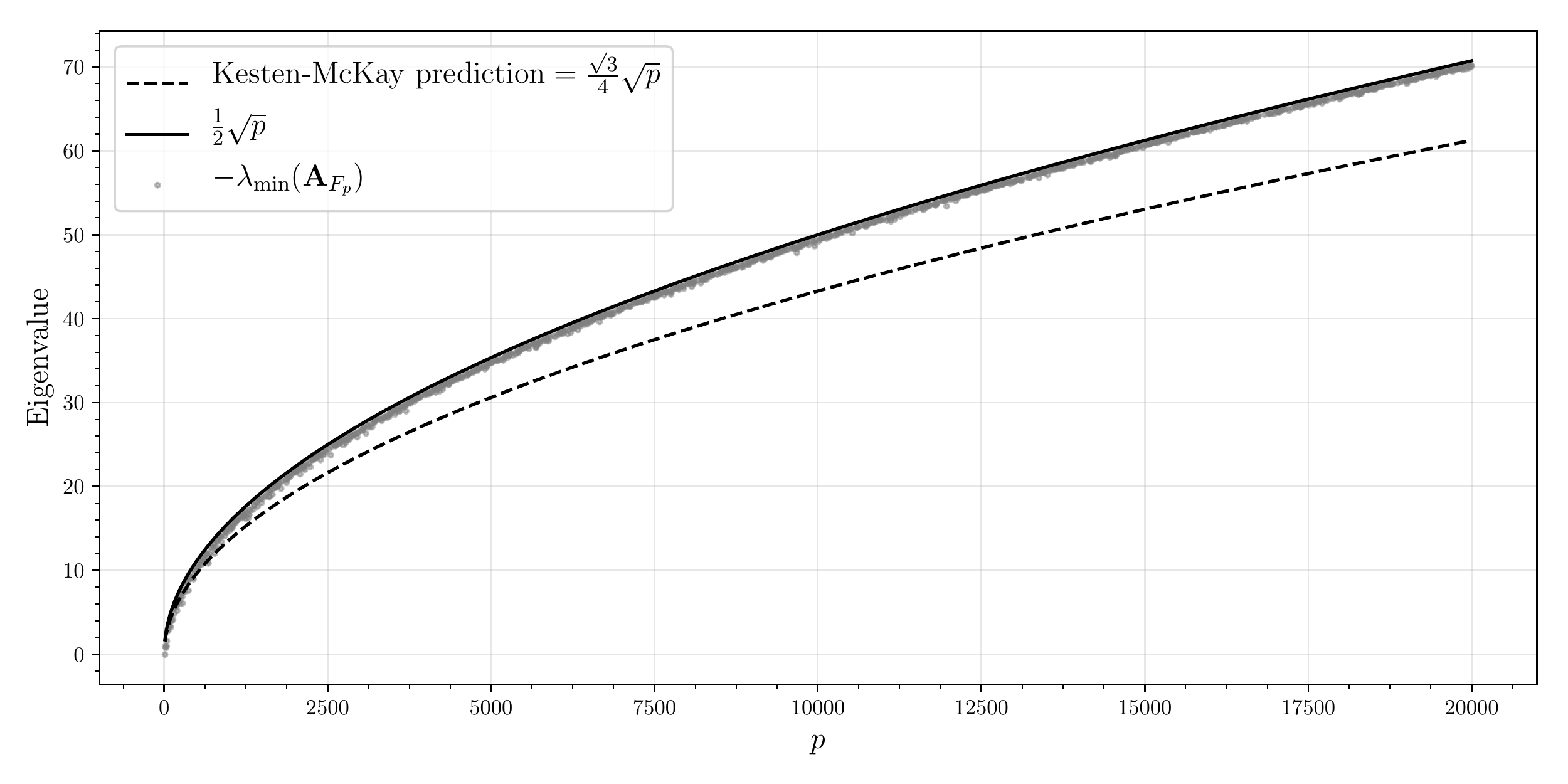}
    \end{center}

    \caption{\textbf{Minimum eigenvalue of $F_p$ subgraph.} We plot the minimum eigenvalue of the induced subgraph $F_p$ of $G_p$ discussed in Appendix~\ref{app:rip} together with the prediction of Conjecture~\ref{conj:min-eval} (for subgraphs induced by localization, unlike $F_p$) and a conjectural larger scaling of $-\frac{1}{2}\sqrt{p}$ that this minimum eigenvalue actually appears to obey closely.}
    \label{fig:rip-min-eval}
\end{figure}

\end{document}